\documentclass[preprint]{imsart}

\usepackage{color}
\usepackage{cancel}
\RequirePackage[OT1]{fontenc}
\RequirePackage{amsthm,amsmath}
\RequirePackage[numbers]{natbib}
\RequirePackage[colorlinks,citecolor=blue,urlcolor=blue]{hyperref}
 
\startlocaldefs
\numberwithin{equation}{section}
\theoremstyle{plain}

\endlocaldefs

\usepackage{amsfonts}
\usepackage{amsmath,color}
\usepackage{amssymb}
\usepackage{url}
\usepackage{latexsym}
\usepackage{amsthm}
 
 \hoffset= - 1.0 cm  \voffset= - 0.1 cm \providecommand{\U}[1]{\protect\rule{.1in}{.1in}}

\newtheorem{theorem}{Theorem}

\newtheorem{hypothesis}{Hypothesis}

\newtheorem{lemma}[theorem]{Lemma}

\newtheorem{proposition}[theorem]{Proposition}
\newtheorem{remark}[theorem]{Remark}

\def\P{\mathbb P}
\def\E{\mathbb E}
\def\N{\mathbb N}

\def\R{\mathbb R}
\def\ds{\displaystyle}
\def\L{\mathcal L}

\def\A{\mathcal A}
\def\H{\mathcal  H}
\def\W{\mathcal  W}

\begin{document}

\begin{frontmatter}
\title{  An optimal  regularity result
 for 
Kolmogorov equations 
  and weak uniqueness 
  for  some   critical 
  SPDEs  }
\runtitle{Some Critical SPDEs}

\begin{aug}
\author{\fnms{Enrico Priola } 
} 
\thanks{The author  have been partially supported by   GNAMPA of the Istituto Nazionale di Alta  Matematica (INdAM).} 

\runauthor{E. Priola}

\address{    Dipartimento di  
 Matematica,  
 \\
Universit\`a  di Pavia, Pavia, Italy
\\
enrico.priola@unipv.it
}

\end{aug}

\vskip 0.5 cm

\begin{abstract}
  We  show uniqueness in law for  the  critical  SPDE  
 $$
 dX_t = AX_t dt + (-A)^{1/2}F(X(t))dt +  dW_t,\;\;   
 X_0 =x \in H,
  $$ 
where   $A$ $ : \text{dom}(A) \subset   H \to H$ is a  negative definite  self-adjoint operator on  a separable Hilbert space $H$ having $A^{-1}$ of trace class  and  $W$ is 
   a cylindrical Wiener process on  $H$. 
 Here 
{  $F: H \to H $  can be    continuous 
  with at most linear growth 
  (some functions  
  $F$ which grow more than linearly 
    can  also be considered)}. This leads to new uniqueness results for 
 generalized stochastic Burgers equations and for three-dimensional   stochastic Cahn-Hilliard type equations 
  which have interesting  applications.    
  To get weak uniqueness we use an infinite dimensional localization principle and also   establish  a new optimal regularity result for 
the   Kolmogorov equation   $ \lambda u -  Lu = f$  
 associated to the SPDE when $F=0$ ($\lambda >0$,   $ f: H \to {\mathbb R}$   Borel and bounded). In particular, we prove that 
  the first derivative $
 Du(x)$ belongs to $\text{dom}((-A)^{1/2})$, for any $x \in H,$ and 
$
\sup_{x \in H} |(-A)^{1/2}Du  (x)|_H 
$ $ =  \| (-A)^{1/2}Du \|_{0} 
 \le 
C \, \| f\|_{0}.
$
\\ 
\end{abstract}


 
\end{frontmatter}

\section{Introduction}

We establish weak uniqueness (or uniqueness in law)    for critical  stochastic evolution equations like 
\begin{equation} \label{sde}
dX_t = AX_t dt + (-A)^{1/2}F(X_t)dt +  dW_t, \;\; X_0 =x \in H.
\end{equation}   
Here $H$ is a separable Hilbert space, $ A: D(A) \subset H \to H$ is a 
self-adjoint operator of negative type 
such that the inverse $A^{-1}$ is of trace class (cf. Section 1.1 and see also Remark \ref{serve}), $W = (W_t)$ is a cylindrical Wiener process on  $H$, cf.  \cite{DZ1},    \cite{DZ},  \cite{hairer}     and the references therein. We {  may}  assume that
 \begin{equation} \label{lin1}
F: H \to H \;\; \text{is continuous and verifies   } \;\; |F(x)|_H \le C_F (1 + |x|_H), \;\; x \in H,
\end{equation}  
 {  for some  constant $C_F >0$. This allows to prove  both weak existence and weak uniqueness for \eqref{sde}.
   Assumption \eqref{lin1} can be   relaxed if we assume  that  weak existence holds for \eqref{sde}; 
  see Section 7 
  where we 
  consider  $F$ which is    continuous and  bounded on bounded sets of $H$.
 }

\smallskip 
 
 Using the analytic  semigroup   $(e^{tA})$   generated by $A$ 
 we consider   mild solutions to \eqref{sde},  i.e., 
\begin{equation*}
  \label{mq1}
X_{t}=e^{tA}x+\int_{0}^{t}(-A)^{1/2}e^{(  t-s)  A}F  (
X_{s})
ds
+\int_{0}^{t}e^{(  t-s)  A}dW_{s},\;\;\; t \ge 0,
\end{equation*}
(cf. Section 1.1) and   prove
  the following result.
\begin{theorem}
\label{base} 
 Under Hypothesis \ref{d1} and assuming \eqref{lin1}, for any $x \in H$,  there exists a weak mild solution defined on some filtered probability space.  
 Moreover uniqueness in law (or weak uniqueness) holds for \eqref{sde}, for any $x \in H$.     
\end{theorem}
Examples of SPDEs of the form \eqref{sde} are considered in Section 2. In particular, we can deal with  stochastic Burgers-type equations like  
 $$
d u (t, \xi)=   \frac{\partial^2}{\partial   \xi^2}  u(t, \xi) {   dt} +   \frac{\partial }{\partial \xi} {  h( } u(t, \xi)){   dt}  + dW_t(\xi), \;\; u(0, \xi) = u_0(\xi), \;\;\; \xi \in (0,\pi),
$$ 
 with suitable boundary conditions (cf. \cite{Gy}, \cite{D2} and \cite{RS}) and stochastic Cahn-Hilliard equations (cf. \cite{EM}, \cite{DD}, \cite{NC}, \cite{ES}) like   
$$
d u (t, \xi)= -  \triangle^2_{\xi} u(t, \xi){   dt} + \triangle_\xi {  h( }u(t, \xi)){   dt} + dW_t(\xi), \;\;t>0,\;\;  u(0, \xi) = u_0(\xi)\;\; \text{on $G$}, 
$$
 with suitable boundary conditions ($G \subset \R^3$ is a regular bounded open set).
 We  prove {\sl weak well-posedness} for both SPDEs when    {  $h$} is continuous and has at most a linear growth {  (see Section 2).} 
  {  Such assumption  does not cover classical stochastic  Burgers equations (i.e., ${   h( }u) = \frac{u^2}{2}$) and stochastic Chan-Hiliard equations (i.e., ${   h( }u) = u^3 -u$) for which strong existence and uniqueness can be proved by different methods (cf. \cite{brezniak} and \cite{DD}).
  On the other hand, in Section 7 we  consider   some    perturbations of classical Burgers equations (cf. Propositions \ref{dap} and \ref{well1}).}

 We  mention that in  \cite{Za},  \cite{bass} and \cite{bass1}     weak uniqueness has been investigated for stochastic evolutions equations   
with H\"older continuous  coefficients  and non-degenerate  multiplicative noise when $(-A)^{1/2} F$ is replaced by $F$.
 On the other hand, 
 weak uniqueness for \eqref{sde} follows by  Section 4 of  \cite{D2} assuming  that $F$ is $\theta$-H\"older continuous and bounded, $\theta \in (0,1)$,  
$
 \text{with} \; \| F\|_{C_b^{\theta}(H,H)} \; $ $\text{small enough. }
$
     To prove weak uniqueness  for \eqref{sde} we first  establish   
a new  optimal regularity result 
 for the infinite-dimensional Kolmogorov equation  
\begin{equation*}
\lambda u (x) - Lu (x)   = f (x),\;\; x \in H. 
\end{equation*}
where $\lambda >0$, $f: H \to \R$ is a given Borel and bounded function and 
 $L$ is  an infinite-dimensional Ornstein-Uhlenbeck operator   which is  formally given by 
\begin{equation*}
 L g(x) = \frac{1}{2} \mbox{Tr}(D^2 g(x)) +
 \langle Ax ,  Dg (x)\rangle, \;\; x \in D(A),
\end{equation*}
where $D g(x)$ and $D^2 g(x)$ denote respectively the
first and second Fr\'echet   derivatives of a regular function $g$ at $x \in H$  and $\langle \cdot, \cdot \rangle$ is the inner product in $H $ {  
(for regularity results concerning $L$  when $H = \R^n$ see \cite{lorenzi} and the references therein).} 
According to Chapter 6 in \cite{DZ1} (see also \cite{D2} and     \cite{DFPR})  we investigate properties of  the bounded  solution $u : H \to \R$,
 \begin{equation}\label{dd9}  
u(x)= \int_0^{\infty} e^{-\lambda t } P_t f(x)dt, \;\;\; x \in H;
 \end{equation}
here $(P_t)$ is the Ornstein-Uhlenbeck semigroup associated to $L$.    One has
$
P_tf(x) $ $= \E[f(Z_t^x)] $ $ = \E \Big[f(e^{tA} x$ $ + \int_0^t e^{(t-s)A}  dW_s)\Big]
$;
    $Z^x$ denotes the Ornstein-Uhlenbeck process which  solves \eqref{sde}
 when $F=0$ (cf. Section 1.2).   
      It easy to prove that  $u \in C^1_b(H) $, i.e., $u$ is continuous and bounded with  the first {   Fr\'echet} derivative $Du : H \to H$ which is continuous and bounded.   
 
 The new regularity result  we prove is  that  $Du(x) \in D((-A)^{1/2})$, for any $x \in H,$   and 
\begin{equation}\label{ss9}         
  \sup_{x \in H}| (-A)^{1/2} Du (x)|_H \, \le \frac{\pi} {\sqrt 2}\,  \sup_{x \in H}| f(x)|_H, 
\end{equation} 
 see Theorem \ref{ss13} with $z=0$ and compare with \cite{bass}, \cite{CL}, \cite{D2} and \cite{LR}.  
  Note that \eqref{ss9}   is a limit case of  known estimates. Indeed  if $\theta \in (0,1)$, and $f : H \to \R $ is $\theta$-H\"older continuous and bounded then 
\begin{equation}\label{ssp} 
 \begin{array}{l}
   \| (-A)^{1/2} Du \|_{C^{\theta}_b(H, H)} \le  c_{\theta} \| f\|_{C_b^{\theta}(H)}
\end{array}    
   \end{equation} 
is the main result in \cite{D2}.  Similar  regularity results  have  been already proved in $L^p(H, \mu)$-spaces with respect to the Gaussian invariant measure $\mu$ for $(P_t)$ (cf. Section 3 of \cite{CG1}):  
      \begin{equation} 
\|(-A)^{1/2}D u\|_{L^{p}(\mu)} \le C_p \|f\|_{L^p(\mu)},    
\;\;\; 1 < p < \infty.
\end{equation}
Hence estimate \eqref{ss9} corresponds to the remaining  case $p=\infty$. 
 We stress that when   $f \in L^2(\mu)$ the fact that  the estimate  $\|(-A)^{1/2}D u\|_{L^{2}(\mu)} \le C_2 \|f\|_{L^2(\mu)}$ is sharp follows by   
  Proposition 10.2.5 in \cite{DZ1}.  
{    Moreover, the  bound 
  \begin{equation} \label{ww5}
  \begin{array}{l}
   \sup_{x \in H} | (-A)^{1/2} D P_t f (x) |
    = 
\| (-A)^{1/2} D P_t f\|_0 \sim \frac{ C_1}{t} \| f\|_0, \;\;\text{as} \;\;  t \to 0^+
\end{array}  
\end{equation}
 ($C_1$ is given in \eqref{e5}) 
 containing 
 the singular term $\frac{1}{t}$ 
 suggests  
 that \eqref{ss9}
 cannot be    improved
 replacing $(-A)^{1/2}$ by  $(-A)^{\gamma}$, $\gamma \in (\frac{1}{2}, 1)$; 
 see  also 
 Chapter 6 of \cite{DZ1} and Remark \ref{mai}.} 

Theorem   \ref{ss13} is deduced by the crucial Lemma \ref{ss1}; the proof of such lemma 
uses the diagonal structure of $A$. By Lemma \ref{ss1} we also       
  derive  another new   regularity result (cf. Theorem \ref{sr} with $z=0$):  
 \begin{equation}\label{ss4} 
\| P_s  Du -  Du \|_0 \le c \, s^{1/2}\| f\|_0,\;\;\; s \in [0,1].
\end{equation}
  In Appendix we show that   \eqref{ss4} implies  the ${\cal C}^1$-Zygmund regularity of $Du$;
  such  Zygmund regularity of $Du$  has been   obtained in \cite{CL} and \cite{LR} by  different methods.    
\def\ciao1{  
 The following  extension  to $L^p(\mu)$, $p>1$, can be found in
Section 3 of \cite{CG1}.    taglialo   ???
 \begin{theorem}
\label{t1cg}
Let $\lambda>0$,    $f\in L^p(H,\mu)$ and let $\varphi\in D(L_p)$ be the solution of the equation
\begin{equation}
\label{e10cg}
\lambda\varphi-L_p\, \varphi=f.
\end{equation}
Then $\varphi\in W^{2,p}(H,\mu)$, $(-A)^{1/2}D\varphi\in
L^{p}(H,\mu;H)$ and there exists  a constant $C= C(\lambda, p)$ such  
that
\begin{equation}
\label{e11cg} \|\varphi\|_{L^p(\mu)} + \Big ( \int_H \|D^2 \varphi
(x)\|_{}^p \, \mu (dx)\Big)^{1/p}
+\|(-A)^{1/2}D\varphi\|_{L^{p}(\mu)}\le C\|f\|_{L^p(\mu)}.
\end{equation}
 \end{theorem}
 } 
 
Concerning the SPDE \eqref{sde}  we first prove   the weak existence in Section 4 (see also Remark \eqref{sd}).
To this purpose we adapt a compactness argument already used in \cite{GG}  (see also Chapter 8 in \cite{DZ}).
The proof of the uniqueness part of Theorem \ref{base} is more involved and it is done in various steps (see Sections 5 and 6). In the case when $F \in C_b(H,H)$ we first  consider   equivalence between mild solutions and solutions to the martingale problem  of Stroock and Varadhan \cite{SV79}     (cf. Section 5.1). This allows to use some  uniqueness results available for the martingale problem (cf. Theorems \ref{ria}, \ref{uni1} and \ref{key}).  On this respect we point out that an infinite-dimensional generalization of the martingale problem is   given in Chapter 4 of \cite{EK}. 
  
   In Section 5.3 we prove   weak uniqueness assuming that 
  there exists   $z \in H$  such that 
  \begin{gather} \label{aa}
\sup_{x \in H}| F(x)- z |_H\,  < 1/4.
\end{gather}
 To this purpose we need a careful  analysis of the Kolmogorov equation 
\begin{gather*}
  \lambda u -Lu -\langle z, (-A)^{1/2} D u 
\rangle= f  + \langle F - z, (-A)^{1/2} D u 
\rangle
\end{gather*} 
 under the condition \eqref{aa} (see Section 5.2). This is based on the fact that the same   estimate  \eqref{ss9} holds more generally  if  $u$ is replaced by
 the       solution  $     u^{(z)}$ of the following equation
 \begin{equation} \label{sdd}
  \lambda u -Lu -\langle z, (-A)^{1/2} D u 
\rangle= f,   
\end{equation} 
 for any $z \in H$ (cf. Theorem \ref{ss13}).   
  In Section 5.4 we prove uniqueness in law when $F \in C_b (H,H)$ (removing condition \eqref{aa}). To this purpose we also adapt the localization principle which has been introduced  in \cite{SV79} (cf. Theorem \ref{uni1}). 
  In Section 6 we complete the proof of Theorem \ref{base}, showing weak     uniqueness under  \eqref{lin1}. To this purpose 
   we truncate $F$ and 
   prove uniqueness  for the martingale problem up to a stopping time (cf. Theorem  \ref{key}). {  Section 7 considers  the case of $F$ which is  continuous and locally bounded.} 
 
  We finally mention 
that recent papers 
  investigate  pathwise uniqueness for  SPDEs like \eqref{sde}
 when  $(-A)^{1/2}F$ is replaced by a measurable drift term  $F$ 
 (cf. \cite{DFPR}, \cite{DFRV}  and  also \cite{mytnik} for the case of semilinear stochastic heat equations and see the references therein). 
 For such equations in infinite dimensions    even if  $F \in C_b(H,H)$   pathwise uniqueness, for any initial  $x \in H$,
 is still not clear
  (however   
 pathwise uniqueness holds for $\mu$-a.e.  $x \in H$). 

\begin{remark} \label{mai}   {\em      It is not clear if our  uniqueness result holds for \eqref{sde} when $(-A)^{1/2}$ is replaced by $(-A)^{\gamma}$, $\gamma \in (1/2,1)$.
 We believe that  for  $\gamma \in (1/2,1)$ there  should exist a continuous and bounded drift $F_{\gamma}: H \to H$ and $x_{\gamma} \in H$ such that weak uniqueness fails for  
  $dX_t = AX_t dt $ $+ (-A)^{\gamma}F_{\gamma} (X_t)dt +  dW_t,$ $  X_0 =x_{\gamma} $ (on the other hand, weak existence holds, cf. Remark \ref{sd}). 
 In this sense \eqref{sde}  can be considered   as a  critical SPDE.}
\end{remark}

\subsection{ Notations and preliminaries}

Let $H$ be a   real separable Hilbert space.  Denote its
norm and inner product  by $\left\vert \cdot \right\vert_H $
 and 
$\left\langle \cdot , \cdot \right\rangle $ respectively. Moreover  ${\mathcal B}(H)$ indicates its Borel $\sigma$-algebra.  
 Concerning \eqref{sde} 
 as in \cite{D2}, \cite{DFPR} and  \cite{DFRV} 
 we  assume  
\begin{hypothesis} \label{d1} 
 $A:D(A)\subset H\to H$ is a negative definite
self-adjoint operator with domain $D(A)$ (i.e., there exists $\omega >0$ such that  $\langle Ax, x\rangle \le - \omega |x|^2_H$, $x \in D(A)$). Moreover
 $A^{-1}$ 
 is a trace class operator.
 \end{hypothesis} 
In the sequel we will concentrate on an infinite dimensional Hilbert space  $H.$ 
 Since $A^{-1}$ is compact, there exists an orthonormal basis $(e_k)$
in $H$ and an infinite sequence of positive numbers $(\lambda_k)$ such that
 \begin{equation} \label{e1a}
 \begin{array}{l}
  Ae_k=-\lambda_k e_k,\quad k\ge 1,\;\; \text{and } \;\; \sum_{k \ge 1} {\lambda_k^{-1}} < \infty.
\end{array}
\end{equation} 
We denote by  ${\mathcal L}(H)$  the Banach space of  bounded and linear operators $T: H \to H$ endowed with the operator norm $\| \cdot \|_{\mathcal  L}.$  
  The operator $A$ generates an analytic semigroup $(e^{tA})$ on $H$
such that $e^{tA} e_k = e^{- \lambda_k t } e_k$, $t \ge 0$.  
 Remark that  
\begin{equation} \label{ewd}
\begin{array}{l}
 \|  (-A)^{1/2} e^{tA} \|_{\cal L} = \sup_{k \ge 1}\, \{  (\lambda_k)^{1/2} e^{-\lambda_k t} \} \le \frac{c}{ \sqrt{t}},\;\;\; t>0,
\end{array}   
 \end{equation}
with  $c= \sup_{u \ge 0} u e^{-u^2}$.  We will also use orthogonal projections with respect to $(e_k)$:
\begin{equation} \label{pp1}
 \begin{array} {l}
\pi_m=\sum_{j=1}^me_j\otimes e_j, \;\;\; 
\pi_m x = \sum_{k=1}^m x^{(k)} e_k, \;\; \text{where $x^{(k)} = \langle x, e_k\rangle $, $x \in H$, $m \ge 1$. }      
 \end{array}
\end{equation} 
   Let $(E, |\cdot |_E)$ be a real separable Banach space. We denote by 
${B}_b(H, E)$
  the Banach space of all real, bounded and Borel functions on
  $H$ with values in $E$, endowed with the supremum norm $\| f \|_0 = \sup_{x \in H} |f(x)|_E$, $f \in {B}_b(H, E).$ Moreover 
  $C_b(H, E) \subset  B_b(H, E)  $ indicates the subspace of all   bounded and continuous  functions.
 We denote by $C^{k}_b (H,E) \subset {B}_b(H, E)$, $k \ge
1$, the  space of all functions $f: H \to E$ which are bounded
and Fr\'echet differentiable on $H$ up to the order $k \ge 1$ with
all the derivatives $D^j f$ bounded and continuous on $H$, $1 \le j \le k$. 
 We also set $B_b(H) = B_b(H,\R)$, $C_b(H) = C_b(H, \R)$ and $C^{k}_b (H, \R)
= C^{k}_b (H)$. 
 
\smallskip   
We  will deal with the SPDE \eqref{sde} 
 where  $W = (W_t)$ $= (W(t))$ is a {\sl cylindrical Wiener} process on $H$.  Thus   
  $W$ is formally given by ``$W_t = \sum_{k \ge 1}
  W^{(k)}_t e_k$'' where  $(W^{(k)})_{k \ge 1}   $   are independent real Wiener processes and   $(e_k)$ is the basis of eigenvectors  of $A$ (cf. \cite{DZ1},  \cite{hairer} and \cite{DZ}).   
 The next definition is meaningful for $F: H \to H$ which is only continuous  because of \eqref{ewd}. 
  
  
 \vskip 1mm 
  A  { \sl  weak mild solution}  to 
(\ref{sde})    is a sequence $( 
\Omega,$ $ {\mathcal F},
 ({\mathcal F}_{t}),$ $ \P, W, X ) $, where $(
\Omega, {\mathcal F},$ $
 ({\mathcal F}_{t}), \P )$ is a  filtered probability space  
  on which it is defined a
cylindrical Wiener process $W$ and
 an ${\cal F}_t$-adapted,   $H$-valued
continuous process $X$ $ = (X_t)$ $ = (X_t)_{t \ge 0}$ such that, $\P$-a.s., 
 \begin{equation}
  \label{mqq}
X_{t}= \,  e^{tA}x \, + \, \int_{0}^{t}(-A)^{1/2}e^{(  t-s)  A}F  (
X_{s})
ds
+\int_{0}^{t}e^{(  t-s)  A}dW_{s},\quad  t \ge 0.
\end{equation}
(hence  $X_0 =x$, $\P$-a.s.). 
 We say that  {\sl uniqueness in law holds for \eqref{sde} for any $x \in H$ } if  given two weak mild solutions   $X$ and $Y$     (possibly defined on different filtered probability spaces and starting at $x \in H$), we have that  $X$ and $Y$ have the same law on ${\cal B}(C([0, \infty); H))$ which is the Borel $\sigma$-algebra of $C([0, \infty); H)$  (this is the  Polish    space of  all continuous functions  from $[0, \infty)$ into $H$ endowed with the metric of the uniform convergence on bounded intervals; cf. \cite{KS} and  \cite{DZ}).
  Note that   the stochastic convolution
$$    
  \begin{array} {l}
W_A(t) =\int_{0}^{t}e^{\left(  t-s\right)  A}dW_{s} = \sum_{k \ge 1}
\int_0^t e^{-(t-s) \lambda_k}e_k dW^{(k)}(s)
\end{array} 
$$
is well defined since 
 the series converges in $L^2(\Omega; H)$, for any $t \ge 0$. Moreover $W_A(t)$ is a Gaussian random variable with values in $H$ with distribution $N(0,Q_t)$
 where 
 \begin{equation}\label{qt1}
 \begin{array} {l}
Q_t = \int_0^t e^{2 sA} ds =  (-2 A)^{-1}(I-e^{2tA}),\;\;\; t \ge 0,
\end{array}
\end{equation}
 is   the covariance operator
(see also \cite{D2}).
Note that $ W_A(t) $ has a continuous version with values in $H$ (see Corollary 2 in  \cite{Tala});   if we assume in addition that $(-A)^{-1 + \delta}$ is of trace class, for some $\delta \in
(0,1)$, then  this fact follows by Theorem 5.11 in \cite{DZ}.

Equivalence between {   different} notions of solutions for \eqref{sde}  
  are clarified  in \cite{DZ1} and \cite{hairer} (see also  \cite{kunze} for a more general setting).    If we write  $ X^{(k)}(t)=  X^{(k)}_t = \langle  X(t), e_k \rangle $, $k \ge 1$, \eqref{sde} is equivalent to the system 
\begin{gather} \label{d33}
   X^{(k)}_t =  x^{(k)} -    \lambda_{k} \int_0^t   X^{(k)}_s ds 
  \, + \, \lambda_{k}^{1/2} \, \int_0^t F^{(k)}( X_s )ds + W_{t}^{(k)},\;\;\; k \ge 1,
\end{gather}     
or to 
 $ \ds  
X^{(k)}_t =  e^{- \lambda_k t}  x^{(k)}    +
  \int_0^t e^{-\lambda_k (t-s)}(\lambda_k)^{1/2}F^{(k)}(X_s)  ds +  \int_0^t e^{-\lambda_k (t-s)} d W^{(k)}_s,
$ 
for $k \ge 1$, $t \ge 0$, with 
$
F(x) = \sum_{k \ge 1}   F^{(k)}(x)e_k, \;\;\; x \in H.   
$

 We will also use the natural filtration of $X$ which is  denoted by  $({\cal F}_t^X)$; 
  ${\cal F}_t^X = \sigma(X_s \, : \, 0 \le s \le t)$ is the  $\sigma$-algebra generated by the r.v. 
 $X_s$, $0 \le s \le t$ (cf. Chapter 2 in \cite{EK}).
\begin{remark} \label{serve} {\em   We point out that Theorem \ref{base}  
   holds under the following  more general hypothesis:
$A:D(A)\subset H\to H$ is self-adjoint, $\langle A x,x  \rangle \le 0$, $x \in D(A)$, and
   $(I -A)^{-1 }$ 
is  of  trace class, with $I = I_H$.
 Indeed in this case  one can  rewrite  equation \eqref{sde} in the form
 $$ 
d X_{t}=(A -I) X_{t}  dt \, 
+ \,  (I - A)^{1/2} [(I -A)^{-1/2} X_t     + (-A)^{1/2} (I -A)^{-1/2}  F(X_{t})] dt \, + \, dW_{t},
$$ 
$ X_{0}=x. $  Now the linear operator   
  $\tilde A = I-A$ and the nonlinear term   $\tilde F(x)=[(I -A)^{-1/2} x         + \, (-A)^{1/2} (I -A)^{-1/2}  F(x)],$
  $ x \in H,$ verify    Hypothesis \ref{d1} and condition    \eqref{lin1} respectively.  
 }
\end{remark}

\subsection {A generalised Ornstein-Uhlenbeck semigroup}    
  
Let us fix $z \in H$.  We will consider  generalised Ornstein-Uhlenbeck operators like  
\begin{equation} \label{ou3}
L^{(z)} g(x) = \frac{1}{2} \mbox{Tr}(D^2 g(x)) +
 \langle x , A Dg (x)\rangle +  \langle z , (-A)^{1/2} Dg (x)\rangle , \; x \in H, \;\; g \in C^2_{cil}(H).
\end{equation}
Here $C^2_{cil}(H)$ denotes the space of {\sl regular 
cylindrical functions.} We say that  $g: H \to \R$ 
   belongs to $C^2_{cil}(H)$ if there exist elements $e_{i_1}, \ldots, e_{i_n}$ of the basis $(e_k)$ of eigenvectors of $A$ and a $C^2$-function  $\tilde g : \R^n \to \R$ with compact support such that 
\begin{gather} \label{cil2}
 g(x) = \tilde g (\langle  x, e_{i_1}\rangle, \ldots, \langle  x, e_{i_n} \rangle),\;\;\; x \in H.
\end{gather}  
By writing the stochastic equation $dX_t = AX_t dt +  (-A)^{1/2} z dt + dW_t,$ $ X_0 =x$ in mild form as  
$
X_t = e^{tA} x $ $+ \int_0^t e^{(t-s)A} dW_s $ $ +  \int_0^t (-A)^{1/2}e^{(t-s)A} z\,  ds, 
$
 one can easily check that the Markov semigroup associated to 
 $L^{(z)}$ is a  generalized Ornstein-Uhlenbeck semigroup $(P_t^{(z)})$:
\begin{gather}\label{gt}
 \begin{array}{l} \ds 
P_t^{(z)} f (x) \, 
 =\, \int_{H} f(e^{tA} x+ y + \Gamma_t z  )  \; {    N  (0 , Q_t)} \, (dy),\;\; 
f \in { B}_b (H),\; x \in H,
\\   \nonumber 
 \text{ setting $\Gamma_t = (-A)^{1/2} \int_0^t e^{sA} ds, $  } 
 \;\;\;\;\; 
 \Gamma_t z = (-A)^{-1/2}[z- e^{tA}z] 
   = \sum_{k \ge 1} \, \frac {(1- e^{-t \lambda_k })}{(\lambda_k)^{1/2}} z^{(k)} \,  e_k. 
 \end{array} 
\end{gather} 
The  case $z=0$. i.e., $(P_t^{(0)})= (P_t)$ corresponds to the well-known Ornstein-Uhlenbeck semigroup (see, for instance, \cite{DZ1}, \cite{DZ}, \cite{D2}, \cite{DFPR} and \cite{DFRV}) which has  a unique invariant measure $\mu =N(0,S)$
where $S=-\frac12\;A^{-1}$. 
  It is also well-known (see, for instance, \cite{DZ1} and  \cite{DZ}) that under Hypothesis \ref{d1}, $(P_t)$ is strong Feller, i.e, $P_t (B_b(H)) \subset C_b(H)$, $t>0$.  
  Indeed we have $e^{tA}(H) \subset Q^{1/2}_t(H)$, $t>0$, or, equivalently,
\begin{equation}
\label{lam}  
\begin{array} {l} 
 \Lambda_t=Q_t^{-1/2}e^{tA}=\sqrt 2\;
(-A)^{1/2}e^{tA}(I-e^{2tA})^{-1/2} \in {\mathcal L}(H),\;\; t>0.
\end{array} 
\end{equation} 
Moreover $P_t (B_b(H)) \subset C_b^{k}(H)$, $t>0$, for any $k \ge 1$. Following the same proof of Theorem 6.2.2 in \cite{DZ1} one can show that 
 under Hypothesis \ref{d1}, for any $z \in H,$ we have 
$P_t^{(z)} (B_b(H)) \subset  C_b^{k}(H)$, $t>0$, for any  $k \ge 1$.
  Moreover, for any $f \in B_b(H)$, $t>0$, the following formula for the directional derivative along a direction $h$ holds: 
\begin{equation}
\label{e3} 
D_h P_t^{(z)} f(x) = 
\langle D P_t^{(z)} f(x),h
 \rangle  = \int_H \langle
  \Lambda_t h,Q_t^{-\frac12} y\rangle \, f (e^{tA}x+y+ \Gamma_t z) \mu_t(dy),
  \; x,  h \in H,
\end{equation}
 where $\mu_t = N(0,Q_t)$ (cf. \eqref{qt1}) and the mapping: $y \mapsto \langle
  \Lambda_t h,Q_t^{-\frac12} y\rangle$ is a centered
   Gaussian random variable
 on $(H, {\cal B}(H),\mu_t)$  with variance $ |\Lambda_t h|^2$
   (cf. Theorem 6.2.2 in \cite{DZ1}). 
   We have 
 \begin{equation} \label{e5}  
 \begin{array}{l}
 \ds   \Lambda_t e_k =\sqrt
2\;(\lambda_k)^{1/2}e^{-t\lambda_k}(1-e^{-2t\lambda_k})^{-1/2} e_k,
\\  \ds   
\;\,  \| \Lambda_t \|_{\cal L}\le C_1
t^{-\frac{1}{2}}, \;\; t>0, \; \; C_1 = \sqrt
2 \cdot \sup_{u \ge 0} { [u \, e^{-u^2}}{(1- e^{-2u^2})^{-1/2}}].
\end{array}
\end{equation}    
 We deduce that,  for $t >0$, $g   \in B_b(H)$, $h,k \in H,$
\begin{equation}\label{wdc}
 \begin{array}{l}
 \| D_h P_t^{(z)} g\|_0 \le \frac{C_1}{\sqrt{t}}|h| \| g\|_0, \;\;\; 
 \| D^2_{hk} P_t^{(z)} g\|_0 \le \frac{\sqrt{2}\, C_1^2}{{t}} \| g\|_0 |h| \, |k|, 
\end{array} 
 \end{equation}
where $D_h P_t^{(z)} g= \langle D P_t^{(z)} g(\cdot),h \rangle $, $D^2_{hk} P_t^{(z)} g
 =  \langle D^2 P_t^{(z)} g(\cdot)h,k \rangle $.     
 
 \smallskip 
  To study equation \eqref{sdd} we will investigate regularity properties of the function
 \begin{equation}\label{wv}
u^{(z)}(x)= \int_0^{\infty} e^{-\lambda t }   P_t^{(z)} f(x)dt, \; \;\; x \in H,\; f \in B_b(H) 
\end{equation}
 (we drop the dependence of $u^{(z)}$ on $\lambda $); see  also the remark below. 
  
\begin{remark} \label{ss} {\em  Let us fix $z \in H$.   We point out that under Hypothesis \ref{d1}
 when  $f \in B_b(H)$ and $x \in H$,  the mapping:  $t \mapsto P_t^{(z)} f(x)$ is right-continuous and bounded on $(0, \infty)$ by the semigroup property and   the strong Feller property. Hence, for any $\lambda>0,$ $u^{(z)}: H \to \R$  given in \eqref{wv}   belongs to $C_b(H)$.    
 \\
Moreover, also
the mapping: $t \mapsto D_h P_t^{(z)} f(x)$ is right-continuous  on $(0, \infty)$, for $x,h \in H
$. To check this fact let us fix $t>0$. Writing $D_{ h} {P_{t+s}^{(z)}}f(x) =   D_{ h} P_{s+ \frac{t}{2}}^{(z)} [P_{t/2}^{(z)} \, f](x)$, $s \ge 0,$ and using the strong Feller property we get easily the assertion.

Since  $\sup_{x \in H}|D P_t^{(z)} f(x)|_H \le \frac{c \| f\|_0}{\sqrt{t} }$, $t>0,$ differentiating under the integral sign, one  shows that there exists 
the directional derivative $D_h u^{(z)}(x)$ at any point $x \in H$ along any direction $h \in H$. 
    Moreover, it is not difficult to prove that  there exists  
the  first {   Fr\'echet}  derivative  $Du^{(z)}(x)$ at any $x \in H$  and  $Du^{(z)} : H \to H$ is continuous and bounded (cf. the proof of Lemma 9 in \cite{DFPR}).  Finally we have the formula
 \begin{equation}\label{uu1}  
\;\; \; D_h u^{(z)}(x) =   
 \int_0^{\infty} e^{-\lambda t } D_h P_t^{(z)} f(x)dt,\;\;\; x,h \in H
\end{equation} 
and the straightforward estimate $\| Du^{(z)} \|_0 \le c(\lambda) \| f\|_0$ with $c(\lambda)$  independent of $z \in H$. We will prove a  better  regularity result for $Du^{(z)}$ in Section 3.
}
\end{remark} 
 
 \begin{remark} \label{maa} {\em 
 In the final part of the  proof of Lemma \ref{ss1} we will  need to use  that
 $\Gamma_t (H ) \subset Q_t^{1/2} (H),$ $  t >0$  
(cf. \eqref{gt}). Note that   this is equivalent to {   saying} that   $ Q_t^{-1/2} \Gamma_t \in {\cal L}(H)$, $t>0$, and    we have
  $Q_t^{-1/2} \Gamma_t $ $
 = \sqrt{2} (I -e^{2tA})^{-1/2}  
 (I -e^{tA}) = \sqrt{2}    (I + e^{tA})^{- 1/2}  
 (I -e^{tA})^{1/2} $ $ \in {\cal L}(H). $   
 } 
 \end{remark}

 \section{Examples}  
\subsubsection{  One-dimensional  stochastic Burgers-type equations      
}   We consider 
  \begin{equation} \label{bur0}
d u (t, \xi)=   \frac{\partial^2}{\partial   \xi^2}  u(t, \xi)dt  +   \frac{\partial }{\partial \xi} {   h( }\xi, u(t, \xi))dt + dW_t(\xi), \;\; u(0, \xi) = u_0(\xi), \;\;\; \xi \in (0,\pi),
 \end{equation}
 with Dirichlet boundary condition $u(t,0) = u(t,\pi)=0$, $t>0$ (cf. \cite{Gy} and \cite{D2} and see the references therein). Here $u_0 \in H = L^2(0,\pi)$ and $A = \frac{d^2}{d  \xi^2}$ with Dirichlet boundary conditions, i.e. $D(A) = H^2(0, \pi) \cap H^1_0 (0, \pi)$. It is well-known that $A$ verifies Hypothesis \ref{d1}. The
eigenfunctions are
 $
 e_k (\xi) = \sqrt{2/\pi} \,  \sin (k \xi), $ $ \xi \in \R,\;\; k \ge 1.
 $
 
 The  eigenvalues
 are $-\lambda_k$, where $\lambda_k = k^2 $. 
  The cylindrical noise is $W_t(\xi) = \sum_{k \ge 1 }  
   W_t^{(k)} e_k(\xi)$ (cf. \cite{DZ}).
   {  Classical stochastic Burgers equations with ${   h( }\xi, u) = \frac{u^2}{2}$ are examples of locally monotone SPDEs and strong uniqueness holds (cf. \cite{brezniak}).}
   In \cite{Gy} {  strong uniqueness  is proved}
  assuming that ${   h( }\xi, \cdot )$ is locally Lipschitz with a linearly growing Lipschitz constant.  
  
  Here we   assume that {\sl  $h: (0, \pi) \times \R \to \R$ is continuous and there exists $C>0$ such that }
 $$
 |{   h( }\xi, s)| \le C (1 + |s|),
 $$
 $s \in \R$, $\xi \in (0, \pi)$ (more generally, one could impose Carath\'eodory type conditions on $h$). 
 
 It is  well-known that the  Nemiskii operator: 
   $x \in H \mapsto {  h( }  \cdot, x(\cdot)) \in H$ is continuous from $H$ into $H$.
 To write \eqref{bur0} in the form \eqref{sde} we define $F: H \to H$ as follows
$$ 
F(x)(\xi) = (-A)^{-1/2}\,  \partial_{\xi}  [{   h( }\cdot , x(\cdot))](\xi),\;\;\; x \in  L^2(0,\pi)= H. 
$$
To check that  $F$ verifies \eqref{lin1} it is enough to prove  that 
$T= (-A)^{-1/2} \, \partial_{\xi}$ can be extended to a bounded linear {   operator} from $L^2(0,\pi)$ into $L^2(0,\pi)$. 
We briefly verify this fact.
 Recall that the domain $D[(-A)^{1/2}]$ coincides with the Sobolev space $ H^1(0, \pi)$.
Take $y \in H^1_0(0, \pi)$ and $x \in L^2(0, \pi)$. Define $x_N = \pi_N x$ (cf. \eqref{pp1}).
 Using that $(-A)^{1/2}$ is self-adjoint and integrating by parts we find (we use inner product in $ L^2(0, \pi)$ and the fact  that $y(0) = y(\pi)=0$)
\begin{gather*}
 \langle (-A)^{-1/2}  \partial_{\xi} \, y , x_N \rangle = \langle   \partial_{\xi} y , (-A)^{-1/2} x_N \rangle =  - \langle    y , \partial_{\xi} (-A)^{-1/2} x_N \rangle.
\end{gather*}   
 Now   $ \partial_{\xi} (-A)^{-1/2} x_N (\xi)    =  \sqrt{2/\pi} \sum_{k=1}^N  x^{(k)} \cos (k \xi)$ and so  $| \partial_{\xi} (-A)^{-1/2} x_N  |_{L^2(0,\pi)}^2$ $=  |x_N  |_{L^2(0,\pi)}^2$. 
  
 It follows that, for any $N \ge 1$,
$| \langle (-A)^{-1/2}  \partial_{\xi} y , x_N \rangle | $ $\le |y|_{L^2(0,\pi)} \, 
 |x |_{L^2(0,\pi)}$
 and we easily get the assertion. 
 Hence $F$ verifies \eqref{lin1} and {\sl  SPDE \eqref{bur0} is  well-posed in weak sense, for any initial condition $u_0 \in L^2(0,\pi)$.}
  
Note that instead of ${   h( }\xi, u)$ one could consider different non local non-linearities like, for instance, $u \, g(|u|_H)$ assuming that $g : \R \to \R$ is continuous and bounded.

\subsubsection{ Three-dimensional stochastic Cahn-Hilliard equations      
}  

 The  Cahn-Hilliard equation is a  model to describe phase separation in a binary alloy and some other media, in the presence of thermal fluctuations; 
  we refer to \cite{NC} for a survey
 on this model. The stochastic Cahn-Hilliard equation
  has been recently much investigated  under monotonicity conditions on $h$ which allow to prove pathwise uniqueness; in one dimension a typical example is ${   h( }s) = s^3 -s$  (see \cite{EM}, \cite{DD}, \cite{NC}, \cite{ES} and the references therein).   
 
 We can treat  such SPDE in one, two or three dimensions. Let us  consider   Neumann boundary conditions in a regular bounded open set $G \subset \R^3$. For the sake of simplicity we concentrate on the cube $G = (0, \pi)^3$.
  The equation has the  form 
 \begin{equation} \label{all}\begin{cases}
d u (t, \xi)= -  \triangle^2_{\xi} u(t, \xi)dt + \triangle_\xi {   h( }u(t, \xi))dt + dW_t(\xi), \;\;t>0,\;\;  u(0, \xi) = u_0(\xi)\;\; \text{on $G$}, 
\\
\frac{\partial}{\partial n} u = \frac{\partial}{\partial n} (\triangle u) =0 \;\; on \; \partial G,
\end{cases}
\end{equation}    
 where $\triangle_{\xi}^2 $ is the bilaplacian and $n$    is the outward unit normal vector on the boundary $\partial G$.  Let us introduce the Sobolev spaces $H^j(G) = W^{j,2}(G)$ and the Hilbert space $H$,
 $$
 \begin{array}{l} 
 H = \big \{ f \in L^2 (G) \, :\, \int_G f(\xi) d\xi =0 \big \}.
 \end{array} 
 $$
 We assume $u_0 \in H$ and  define $ D(A) = \{ f \in H^4 (G) \cap H \, :\,  \frac{\partial}{\partial n} f = \frac{\partial}{\partial n} (\triangle f) =0 $    on $ \partial G \big \}, $ $A f = - \triangle^2_{\xi} f$, $f \in D(A)$. Using also the divergence theorem, we have $A : D(A) \to H$.
 
 The square root has domain $D[(-A)^{1/2}]   = \{ f \in H^2 (G) \cap H \, :\,  \frac{\partial}{\partial n} f  =0 $  on $ \partial G\big \};$ $- (-A)^{1/2} f = \triangle_{\xi} f$, $f \in D[(-A)^{1/2}]$. 
 Note that 
 $A$   is self-adjoint with compact resolvent  and it is   negative definite with $\omega =1$ (cf. Hypothesis \ref{d1}). 
  The
eigenfunctions are
 $$
 e_k (\xi_1, \xi_2, \xi_3) = (\sqrt{2/\pi})^3 \cos (k_1 \xi_1) \cos (k_2\xi_2)
 \cos(k_3 \xi_3), \;\;\; \xi = (\xi_1, \xi_2, \xi_3) \in
 \R^3,  $$
 $k=(k_1, k_2, k_3) \in
 \N^3$, $k \not = (0,0,0)= 0^*$.
 The corresponding  eigenvalues  
 are $-\lambda_k$, where $\lambda_k =  (k_1^2 + k_2^2 + k_3^2 )^2$.
 Since $\sum_{k  \in \N^3, \, k \not = 0^*} \, \lambda_k^{-1} < + \infty 
 $  we see that $A$ verifies Hypothesis \ref{d1}. 
The cylindrical Wiener process  is     $W_t(\xi) = \sum_{k \in \N^3,\; k \not = 0^*}  
   W_t^{(k)} e_k(\xi)$. 
   Note that  $\triangle_\xi {   h( }u(t, \xi)) =  \triangle_\xi \big [ 
  {   h( }u(t, \xi)) - \int_G  {   h( }u(t,\xi))d\xi\big] $.
     
Assuming that {\sl $h: \R \to \R$ in \eqref{all} is continuous and verifies $|{   h( }s)| \le c(1 + |s|)$, $s \in \R$,} we can define $F : H \to H$ as follows:
$$
\begin{array}{l} 
F(x)(\xi) = {   h( }x(\xi)) - \int_G  {   h( }x(\xi))d\xi,\;\;\; x \in H,\; \xi \in G.  
\end{array} 
$$
 It is not difficult to prove that  $F$ verifies \eqref{lin1}. Thus {\sl SPDE \eqref{all} is well-posed in weak sense, for any initial condition $u_0 \in H$.  
 }
 
 \section {A new  optimal  regularity result  }

Let $f \in B_b(H)$ and fix $z \in H$. Here we are interested in the regularity property of the function $u^{(z)}: H \to \R$ given in \eqref{wv}. 
  By Remark \ref{ss} we know that  $u^{(z)} \in C^1_b(H)$ and we have a formula for the directional derivative:  
\begin{gather} \label{s55} 
  D_h u^{(z)}(x) = \langle Du^{(z)}(x), h \rangle = \int_0^{\infty} e^{- \lambda t}  D_{ h} P_t^{(z)} f(x) dt,\;\;\; x,h \in H, \; \lambda >0. 
\end{gather}
 In the sequel,  for any $s \ge 0,$ we will consider the bounded   linear  operator $(-A)^{-1/2} e^{sA} : H \to H$  defined as 
 $$
 \begin{array}{l}
  (-A)^{-1/2} e^{sA} h = \sum_{k \ge 1} \lambda_k^{-1/2} e^{-s \lambda_k}\, h^{(k)} e_k,
 \end{array}
 $$
 $h \in H$.
 The following  lemma will be  important.  
\begin{lemma}\label{ss1} 
 For $s \ge 0$, $f \in B_b(H)$, $\lambda >0$, $x, h, z \in H$, we have
$$
 \Big| \int_0^{\infty} e^{- \lambda t}  D_{ h} {P_{t+s}^{(z)}}f(x) dt \Big| 
 \le 
  \frac{\pi} {\sqrt 2} 
  \| f\|_0 \, |(-A)^{-1/2} e^{sA} h|_H.
 $$
  \end{lemma}
\begin{proof} Let us fix $f \in B_b(H)$, $s \ge 0 $, $\lambda >0$, $z$ and $x \in H$. We proceed in two steps.

\noindent {\it I Step. We consider $f \in C_b(H)$.}  We know that the functional $R_{x,f,s, \lambda}^{(z)} : H \to \R$,
\begin{gather} \label{s11}
 R_{x,f,s, \lambda }^{(z)} (h) = \int_0^{\infty} e^{- \lambda t}  D_{ h} P_{t+s}^{(z)} f(x) dt =\int_0^{\infty} e^{- \lambda t}  \langle D {P_{t+s}^{(z)}}f(x), h \rangle dt,\;\; h \in H, 
\end{gather}
is linear and bounded, thanks to the estimate $\sup_{x \in H} |D {P_r^{(z)} } f(x)|_H \le \frac{c \| f\|_0 }{\sqrt{r}}$, $r>0$. 
  Recall   the projections $ \pi_N  : H \to \text{Span}\{e_1, \ldots, e_N \}$ (see \eqref{pp1});  $\pi_N h = \sum_{k=1}^N h^{(k)} e_k,$ $N \ge 1$.
The assertion for $f \in C_b(H)$ follows if we prove that, for any $h \in H,$ $N \ge 1,$  
\begin{equation} \label{dbb}
  |R_{x,f,s, \lambda }^{(z)} (\pi_N h)| \le C |(-A)^{-1/2} e^{sA}  h|_H \, \| f\|_0,
\end{equation}
where $C$ is independent of $N$ and $h$. We   fix $h \in H$ and  define 
$$
\begin{array}{l}
\pi_N h = h_N.
\end{array}  
$$
 We first 
 assume that $f \in C_b (H)$ 
 depends only  on a finite numbers of coordinates.  Identifying
  $H$ with $l^2({\mathbb N})$ through the basis $(e_k)$, we  have 
$$
f (x) = \tilde f(x^{(1)}, \ldots,
  x^{(m)}), \;\;\; x \in H,
$$  
 for some  $m \ge 1$, and  $\tilde f : \R^m \to \R$ continuous and bounded. 

Setting $Q_t e_k = Q_t^k e_k$, where 
  $Q_t^k = \int_0^t e^{-2 \lambda_k r} dr = \frac{1 - e^{-2 \lambda_k t}}{2 \lambda_k}$  (cf. \eqref{qt1}) we consider
\begin{gather*}  
\begin{array}{l}
p(x) = \frac{1}{\sqrt{2 \pi}}e^{- x^2/2 },\;\; c_k(t) =
\Big (\frac{ 1- e^{-2\lambda_k t} } {2 \lambda_k}  \Big)^{1/2} = \sqrt{Q_t^k}.
\end{array}  
\end{gather*}
Note that, for $t>0,$ the one dimensional Gaussian measure $N(0, \int_0^t e^{-2 \lambda_k s}ds)$ has density $p(\frac{ x}{ c_k(t)} ) \frac{1}{c_k(t)}$ with respect to the Lebesgue measure on $\R.$     
 Recall also that  
 \begin {equation}\label{reca}
\Lambda_{t+s}^k = \sqrt
2\;(\lambda_k)^{1/2}e^{-[t+s]\lambda_k}(1-e^{-2[t+s]\lambda_k})^{-1/2},
\end{equation} 
$ t > 0, \; k \ge 1$ (see \eqref{lam}). In the sequel we concentrate on the more difficult case    $N > m$ (if $N \le m$ we can obtain \eqref{dbb} arguing similarly). By \eqref{e3} we find, integrating over $\R^N$, 
\begin{gather*}  
  R_{x,f,s, \lambda }^{(z)} (h_N) =
\int_0^{\infty} e^{- \lambda t}\Big (   \int_H {f}(e^{(t+s)A}x + y + 
\Gamma_{t+s}z ) \,  
 \Big[\sum_{k=1}^N \Lambda_{t+s}^k \frac{ h^{(k)}\, y_k} {c_k(t+s)} \Big]
 N(0, Q_{t+s})  dy \Big ) dt  
\\
= \int_0^{\infty} e^{- \lambda t} dt \int_{\R^N}  {\tilde f}\big (e^{-(t+s)\lambda_1}x^{(1)} + y_1 + \frac{1- e^{-(t+s) \lambda_1 } }{\sqrt{ \lambda_1} } z^{(1)}, \ldots 
\\ \ldots, e^{-(t+s)\lambda_m}x^{(m)} + y_m + \frac{1- e^{-(t+s) \lambda_m }}{\sqrt{\lambda_m}}       z^{(m)} \big) \,  \cdot 
\\
 \cdot
 \Big[\sum_{k=1}^m \Lambda_{t+s}^k \frac{ h^{(k)} \, y_k} {c_k(t+s)}  + 
 \sum_{k=m+1}^N \Lambda_{t+s}^k \frac{ h^{(k)} \, y_k} {c_k(t+s)} 
  \Big] \prod_{k=1}^N  p \big(\frac{ y_k} {c_k(t+s)} \big)\frac{1}{{c_k(t+s)}}
   dy_{1} \ldots   dy_{N}.
\end{gather*}
 Set 
  $
 \Gamma_k(r) = \frac{[1- e^{-r \lambda_k }]}{(\lambda_k)^{1/2}},   \;\;\; k \ge 1, \, r \ge 0.
 $ 
 Since by the Fubini theorem
 $$
 \begin{array}{l}     
 0= \int_{\R^{N-m}} \Big[\sum_{k=m+1}^N \Lambda_{t+s}^k \frac{ h^{(k)} \, y_k} {c_k(t+s)} \Big] \prod_{k=m+1}^N  p (\frac{ y_k} {c_k(t+s)} )\frac{1}{{c_k(t+s)}}
  dy_{m+1} \ldots   dy_{N}  
 \; \cdot 
\\  \ds \cdot \int_{\R^m}  {\tilde f}(e^{-(t+s)\lambda_1}x^{(1)} + y_1 + \Gamma_1(t+s) z^{(1)}, \ldots, e^{-(t+s)\lambda_m}x^{(m)}   + y_m + \Gamma_m(t+s) z^{(m)}) \, \cdot 
\\ \ds \, \cdot \,    
  \prod_{k=1}^m  p (\frac{ y_k} {c_k(t+s)} )\frac{1}{{c_k(t+s)}}
  dy_1 \ldots dy_m,  
  \end{array}  
 $$ 
  we find 
  \begin{gather*}
 R_{x,f,s, \lambda }^{(z)} (h_N) = \int_0^{\infty} e^{- \lambda t}  D_{ h_N} {P_{t+s}^{(z)}} f (x) dt  
\\
= \int_0^{\infty} e^{- \lambda t} dt \int_{\R^m}  {\tilde f}(v_1, \ldots , v_m) \,  \cdot 
 \Big[\sum_{k=1}^m \Lambda_{t+s}^k \frac{ h^{(k)} \, [v_k - e^{-(t+s)\lambda_k}\, x^{(k)} - \Gamma_k(t+s) z^{(k)} ]} {c_k(t+s)}   
  \Big] \cdot
  \\ \cdot \, \prod_{k=1}^m  p \big(\frac{ v_k - e^{-(t+s)\lambda_k}\, x^{(k)} -  \Gamma_k(t+s) z^{(k)} } {c_k(t+s)} \big)\frac{1}{{c_k(t+s)}}
   dv_{1} \ldots   dv_{m}. 
\end{gather*}    
In the sequel to simplify the notation we write 
$$
\tilde h=(h^{(1)}, \ldots,  h^{(m)}), \;\;\; 
\tilde x =(x^{(1)}, \ldots,
  x^{(m)}) ,
$$  $\tilde v = (v_1, \ldots,
  v_m)$, $\Lambda_{t+s} \tilde h = (\Lambda_{t+s}^1 h^{(1)}, \ldots,
 \Lambda_{t+s}^m h^{(m)}) \in \R^m $. By the Fubini theorem,  
 we deduce
 \begin{gather*}
 R_{x,f,s, \lambda }^{(z)} (h_N)   
 = \int_{\R^m}  {\tilde f}(\tilde v ) d \tilde v  \int_0^{\infty} e^{- \lambda t}  
 \Big[\sum_{k=1}^m \Lambda_{t+s}^k \frac{ h^{(k)} \, [v_k - e^{-(t+s)\lambda_k}x^{(k)} -
 \Gamma_k(t+s) z^{(k)}]} {c_k(t+s)}   
  \Big] 
 \\ \, \cdot \, 
  \prod_{k=1}^m  p \big(\frac{ v_k - e^{-(t+s)\lambda_k}x^{(k)} - \Gamma_k(t+s) z^{(k)} } {c_k(t+s)} \big)\frac{1}{{c_k(t+s)}} dt.
\end{gather*}
Now, for any fixed $\tilde v \in \R^m$, we have, recalling \eqref{reca} and   changing variable: $u = \lambda_k t$,
\begin{gather*}
 \int_0^{\infty} e^{- \lambda t}   \Big[\sum_{k=1}^m \Lambda_{t+s}^k \frac{ h^{(k)} \, [v_k - e^{-(t+s)\lambda_k}x^{(k)} - \Gamma_k(t+s) z^{(k)}] } {c_k(t+s)}   
  \Big] 
\\ \, \cdot \,   
  \prod_{k=1}^m  p \big(\frac{ v_k - e^{-(t+s)\lambda_k}{x^{(k)}} - \Gamma_k(t+s) {z^{(k)}} } {c_k(t+s)} \big)\frac{1}{{c_k(t+s)}} dt
\\  =  \sqrt
{2} \int_0^{\infty} {     \sum_{k=1}^m }\, \, e^{- \lambda u/\lambda_k }  \frac{  e^{-u}}{ (1-e^{-2u }  e^{-2s\lambda_k})^{1/2}}  \cdot \\ \cdot \;     (\lambda_k)^{-1/2} e^{-s\lambda_k}  h^{(k)} \,
\frac{  [v_k - e^{-(u/ \lambda_k \, +s)\lambda_k}{x^{(k)}} - \Gamma_k(u/ \lambda_k \, +s) {z^{(k)}} ] } {c_k(u/ \lambda_k \, +s)} \, \cdot 
\\
\cdot  \prod_{k=1}^m  p \big(\frac{ v_k - e^{-(u/\lambda_k \, +s)\lambda_k}{x^{(k)}}  - \Gamma_k(u/ \lambda_k \, +s) {z^{(k)}}} {c_k(u/\lambda_k  \, +s)} \big)\frac{1}{{c_k(u/\lambda_k  \, +s)}} du.
\end{gather*}
By the Fubini theorem and using $1-e^{-2u} e^{-2 s \lambda_k} \ge 1-e^{-2u} $,  $u \ge 0$, $k \ge 1$, we get 
\begin{equation} \label{wcv1} 
\begin{array}{l} 
\ds 
   |R_{x,f,s, \lambda }^{(z)} (h_N) | 
      =   \sqrt
{2}   
 \Big | \int_0^{\infty} du \int_{\R^m} {   \sum_{k=1}^m } \, \,  f(\tilde v)  e^{- \lambda u/\lambda_k }  \frac{  e^{-u}}{ (1-e^{-2u }  e^{-2s\lambda_k})^{1/2}}  \,  
 \\ \ds \, \cdot \, 
   (\lambda_k)^{-1/2} e^{-s\lambda_k}  h^{(k)} \,
\frac{  [v_k - e^{-(u/ \lambda_k \, +s)\lambda_k}{x^{(k)}} - \Gamma_k(u/ \lambda_k \, +s) {z^{(k)}}] } {c_k(u/ \lambda_k \, +s)} \, 
\\ \ds  \, \cdot \,  \prod_{k=1}^m  p \big(\frac{ v_k - e^{-(u/\lambda_k \, +s)\lambda_k}{x^{(k)}} - \Gamma_k(u/ \lambda_k \, +s) {z^{(k)}} } {c_k(u/\lambda_k  \, +s)} \big)\frac{1}{{c_k(u/\lambda_k  \, +s)}} d \tilde v   \Big|
\\ \ds
\le  \sqrt
{2}  \int_0^{\infty}  \frac{ e^{-u}} 
{(1-e^{-2u }  )^{1/2}}   du \int_{\R^m} |f(\tilde v)| \,  \cdot
\\ \ds   \cdot   \Big | \sum_{k=1}^m   \frac{ e^{-s\lambda_k}} { (\lambda_k)^{1/2}}  h^{(k)} \, \frac{  [v_k - e^{-(u/ \lambda_k \, +s)\lambda_k}{x^{(k)}} - \Gamma_k(u/ \lambda_k \, +s) {z^{(k)}} ]} {c_k(u/ \lambda_k \, +s)}   \Big| \, \,  \cdot 
\\ \ds  
 \cdot \, \prod_{k=1}^m  p \Big(\frac{ v_k - e^{-(u/\lambda_k \, +s)\lambda_k}{x^{(k)}}  - \Gamma_k(u/ \lambda_k \, +s) {z^{(k)}}} {c_k(u/\lambda_k  \, +s)} \Big)\frac{1}{{c_k(u/\lambda_k  \, +s)}} d \tilde v.
\end{array}
\end{equation} 
Let us fix $u \ge 0$; we have, changing variable in the integral over $\R^m$,
\begin {gather*}
\int_{\R^m} |f(\tilde v)| \,     \Big | \sum_{k=1}^m   (\lambda_k)^{-1/2} e^{-s\lambda_k}  h^{(k)} \, \frac{  [v_k - e^{-(u/ \lambda_k \, +s)\lambda_k}{x^{(k)}} - \Gamma_k(u/ \lambda_k \, +s) {z^{(k)}} ] } {c_k(u/ \lambda_k \, +s)}   \Big| \, \,  \cdot 
\\
 \cdot \, \prod_{k=1}^m  p \Big(\frac{ v_k - e^{-(u/\lambda_k \, +s)\lambda_k}{x^{(k)}}  - \Gamma_k(u/ \lambda_k \, +s) {z^{(k)}}} {c_k(u/\lambda_k  \, +s)} \Big)\frac{1}{{c_k(u/\lambda_k  \, +s)}} d \tilde v
\\
 \le \| f\|_0
\int_{\R^m}   \Big | \sum_{k=1}^m   (\lambda_k)^{-1/2} e^{-s\lambda_k}  h^{(k)} {  y_k  }   \Big| \, \, 
 N(0,I_m)
 (d\tilde y), 
\end{gather*}
using  the standard Gaussian measure $N(0,I_m)$ with density $\prod_{k=1}^m  p ( y_k)$.  We find   
\begin{gather*} 
\begin{array}{l}
\ds \int_{\R^m}   \Big | \sum_{k=1}^m   (\lambda_k)^{-1/2} e^{-s\lambda_k}  h^{(k)} {  y_k  }   \Big| \, \, 
    N(0,I_m)
 (d\tilde y)
  \\ \ds \le  
 \big( \int_{\R^m}    \sum_{k=1}^m   (\lambda_k)^{-1} e^{-2s\lambda_k}  [h^{(k)}]^2 {  y_k^2  }    \, \, 
   N(0,I_m) (d\tilde y) \big)^{1/2}
    =   \, \big ( \sum_{k=1}^m   (\lambda_k)^{-1} e^{-2s\lambda_k}  [h^{(k)}]^2 \big)^{1/2}.
  \end{array}   
\end{gather*}
By \eqref{wcv1}   it follows that  
$$ \begin{array}{l} \ds
|R_{x,f,s, \lambda }^{(z)} (h_N) | 
\le \sqrt{2} \| f\|_0 \, \Big ( \sum_{k=1}^m   (\lambda_k)^{-1} e^{-2s\lambda_k} \, [h^{(k)}]^2 \Big)^{1/2}  \, \int_0^{\infty}  e^{-u} 
(1-e^{-2u }  )^{-1/2} du    
\\ \ds
\le \frac{\pi} {\sqrt 2}  \| f\|_0 \,  |(-A)^{-1/2} e^{sA} h|_H
\end{array}
$$
and so \eqref{dbb} holds.  Now   we treat  an arbitrary $f \in C_b (H)$.
We introduce  the cylindrical functions $f_n,$ $
 f_n (x)  $ $= {   f(\pi_n x)} $ $ = f \Big( \sum_{k =1}^n x^{(k)} e_k \Big),
 $ $ n \in \N, $ $ x \in
 H.
$  It is clear that $\| f_n \|_0 \le  \| f\|_0$, $n \in
\N$, and moreover $f_n (x) \to f(x)$, for any $x \in H$.

By the previous estimate with $f$ replaced by $f_n$,
we get 
  \begin {equation} \label{www1}
 \Big| \int_0^{\infty} e^{- \lambda t}  D_{ h} {P_{t+s}^{(z)}}f_n (x) dt \Big|
 \le 
   \frac{\pi} {\sqrt 2}  \| f\|_0 |(-A)^{-1/2} e^{sA} h|_H,\;\;\; h \in H.
 \end{equation}
Let $t>0$ (recall that $s \ge 0$ is fixed). According to \eqref{e3} we have 
\begin{gather} \label{s331} 
 D_{ h} {P_{t+s}^{(z)}}f_n(x) = \int_H \langle
  \Lambda_{t+s} h, Q_{t+s}^{-\frac12} y\rangle\, f_n(e^{(t+s)A}x+y +\Gamma_{t+s}z  )\, N(0,Q_{t+s})(dy);
\end{gather}
 we can easily pass to the limit as $n \to \infty$ in \eqref{s331} by the Lebesgue convergence theorem and get $D_{ h} {P_{t+s}^{(z)}}f_n(x) \to $  $D_{ h} {P_{t+s}^{(z)}}f(x)$. 
  Similarly,  using also the estimate $ |D {{P_{t+s}^{(z)}}} f(x)|_H \le \frac{c \| f\|_0 }{\sqrt{t}}$, $t>0$, we can pass 
to the limit, as $n \to \infty$, in \eqref{www1} and  
 obtain \eqref{dbb} when $f \in C_b(H).$
 
 \noindent {\it II Step. Let us consider  $f \in B_b(H)$.} 

Here we use  the invariant measure $\mu = N(0,-\frac12 A^{-1})$ for $(P_t)$ (i.e., $(P_t^{(z)})$ when $z=0$).
 There exists a uniformly bounded   sequence $(f_n) \subset C_b(H) $ such that $f_n(x) \to f(x)$, as $n \to \infty$, for any $x \in H$, $\mu$-a.s., and $\| f_n\|_{0} \le \| f\|_0$ (to this purpose it is enough to note that $P_{1/k} f \to f $ in $L^2(H,\mu)$ as $k \to \infty$). 

It is well-known that for any $r>0,$ $x \in H,$ $N(e^{rA} x, Q_{r})$ is equivalent to $\mu$ (see \cite{DZ} and \cite{DFPR}).  This follows from the fact that $(P_t)$ 
is strong Feller and irreducible and so we can apply the  Doob theorem (cf. Proposition 11.13 in \cite{DZ}).
We note that,  for  $r>0$, $x \in H,$  
$$
| {P_r^{(z)} } f_n(x) -  {P_r^{(z)} } f(x)|
\le  \int_{H} | f_n ( y) - f(y) |\,  N(e^{rA}x+ \Gamma_r z , Q_r)dy.   
$$
By Remark \ref{maa}  we know  $\Gamma_r z \in Q_r^{1/2} (H)$, where   $Q_r^{1/2} (H)$ is the Cameron-Martin space of $N(e^{rA}x , Q_r)$. Applying the Feldman-Hajek theorem we find that $N(e^{rA}x+ \Gamma_r z , Q_r) $ and $N(e^{rA}x, Q_r)$ are  equivalent.
 
 By the Lebesgue  theorem we find, for  $x \in H,$     $r>0,$
  $ {\lim_{n \to \infty}| {P_r^{(z)} } f_n(x) -  {P_r^{(z)} } f(x)|=0.}$
Now let us fix $t>0$. Writing $D_{ h} {P_{t+s}^{(z)}}f_n =   D_{ h} P_{s+ \frac{t}{2}}^{(z)} P_{t/2}^{(z)} \, f_n$, we have
$$
 D_{ h} {P_{t+s}^{(z)}}f_n(x) = \int_H \langle
  \Lambda_{\frac{t}{2}+ s} h, Q_{\frac{t}{2} +s}^{-\frac12} y\rangle\, P_{\frac{t}{2}}^{(z)}f_n(e^{(\frac{t}{2} +s)A}x+y +  \Gamma_{\frac{t}{2}+s}z)\, N(0,Q_{\frac{t}{2}+s})(dy).
$$ 
Since   
 $P_{t/2}^{(z)} f_n(x) \to P_{t/2}^{(z)} f(x)$, for any $x \in H$, we find easily $ D_{ h} {P_{t+s}^{(z)}}f_n(x) 
\to  D_{ h} P_{s+t/2 }^{(z)} P_{t/2}^{(z)} f(x)$ $=  D_{ h} {P_{t+s}^{(z)}}(x)$, as $n \to \infty.$ Since, for any $n \ge 1$, $h \in H$,
  \begin {equation} \label{www12}
  \begin{array}{l}
\Big| \int_0^{\infty} e^{- \lambda t}  D_{ h} {P_{t+s}^{(z)}}  f_n (x) dt \Big| \le 
   \frac{\pi} {\sqrt 2}\,  \| f\|_0 \, |(-A)^{-1/2} e^{sA} h|_H,
\end{array}  
 \end{equation}
  passing to the limit as $n \to \infty $ (using also \eqref{wdc}) we obtain the assertion.
\end{proof}

By considering $s=0$ in the previous lemma, we obtain 
\begin{theorem}\label{ss13} Let $f \in B_b(H)$, $\lambda >0$, $z \in H$ and consider $u^{(z)} \in C^1_b(H)$ given in \eqref{wv}.
 The following assertions hold.  
\\
(i) For any $x \in H$ we have
$Du^{(z)}(x) \in D((-A)^{1/2})$, $(-A)^{1/2} Du^{(z)} : H \to H $ is Borel and bounded and 
\begin{equation}\label{w44}    
 \sup_{x \in H} \big [\sum_{k \ge 1} \lambda_k  (D_{e_k} u^{(z)} (x))^2  \big]^{1/2}= \| (-A)^{1/2} Du\|_0  \le \frac{\pi} {\sqrt 2} \|f\|_0.
\end{equation}
(ii) Let $(f_n) \subset B_b(H)$ be such  that $\sup_{n \ge 1 }\| f_n\|_0 \le C < \infty$ and  $f_n(x) \to f(x)$, $x \in H$. Define
$$
u^{(z)}_n(x) = \int_0^{\infty} e^{- \lambda t}  P_{t}^{(z)}  f_n(x) dt.
$$
Then, for any $h \in H$,  
 \begin{equation} \label{s1144} 
\langle (-A)^{1/2} Du^{(z)}_n(x), h \rangle  
\to  
\langle (-A)^{1/2} Du^{(z)}(x), h \rangle,    \;\; \text{as $n \to \infty$, $x, z  \in H$.}    
\end{equation}
\end{theorem}  
\begin{proof} (i) Set $B = (-A)^{1/2} $ with domain $D(B)$. Let us fix $x,  z  \in H$ and recall \eqref{s55}. 

We know by Lemma \ref{ss1} with $s=0$ that $k_0 = Du^{(z)}(x) $ verifies $|\langle k_0, B  h'\rangle| \le \frac{\pi} {\sqrt 2} \| f\|_0 | h'|_H$, $h' \in D(B)$. 
This implies that $k_0 \in D(B^*)$ and $|B^* k_0|_H \le  \frac{\pi} {\sqrt 2} \| f\|_0$. The assertion follows since  $B$ is  self-adjoint. Moreover, since $\sum_{k=1}^N \sqrt{\lambda_k}\,  D_{e_k} u^{(z)} e_k $ converges pointwise to $(-A)^{1/2} Du^{(z)}$ as $N \to \infty$ 
 we get the desired Borel measurability. 

 \smallskip 
(ii) To prove \eqref{s1144} we first note that, for $t>0$, $ D_h {P_{t}^{(z)} } f_n(x) \to D_h {P_{t}^{(z)} } f(x)$, as $n \to \infty, $ $x,h \in H$ 
 (see the argument after  formula \eqref{s331}). Moreover, $|D {{P_{t}^{(z)}}} f_n(x)|_H \le \frac{c \| f_n\|_0 }{\sqrt{t}}$ $\le \frac{c \,C}{\sqrt{t}}$, 
 $t>0$. Hence we have, for any $k \in D( (-A)^{1/2} )$, $x \in H, $
$$
 \lim_{n \to \infty} \langle Du^{(z)}_n(x) ,(-A)^{1/2} k \rangle    = 
  \langle Du^{(z)} (x) ,(-A)^{1/2} k \rangle.  
$$
By the first assertion  we deduce  that $
 \lim_{n \to \infty} \langle (-A)^{1/2} Du^{(z)}_n(x) , k \rangle    = 
  \langle (-A)^{1/2} Du^{(z)} (x) , k \rangle. $ It follows easily that 
 \eqref{s1144} holds, for any $h \in H.$
\end{proof}
\begin{remark} {\em 
Note that if $G \in B_b(H,H)$, then \eqref{s1144} implies that
\begin{equation}
\label{w12}
\lim_{n \to \infty} \langle (-A)^{1/2} Du^{(z)}_n(x) , G(x)\rangle =
  \langle (-A)^{1/2}   Du^{(z)}(x) , G(x)\rangle,\;\; x,z \in H. 
\end{equation}
This fact will be useful in the sequel. \qed   
}
\end{remark}
Recall  that 
\begin{equation*} 
 u^{(z)}(x) = \int_0^{\infty} e^{- \lambda t}   {P_{t}^{(z)} } f(x) dt =
 R^{(z)}(\lambda) f(x)
 \end{equation*}  
where the resolvent $R^{(z)}(\lambda): B_b(H) \to B_b(H)$ verifies the identity
\begin{gather*}
R^{(z)}(\mu) - R^{(z)}(\lambda) =  (\lambda - \mu) R^{(z)}(\mu) \, R^{(z)}(\lambda),\;\; \lambda, \mu>0.
\end{gather*}
 By    Lemma \ref{ss1} we can also obtain the following  new   regularity result.  
  It implies ${\cal C}^1$-Zygmund regularity for $Du$; see Appendix  
  (in finite dimension for Ornstein-Uhlenbeck semigroups such implication  
  is proved in  Lemma 3.6 and Proposition 3.7 with $\theta =1/2$; see also Remark \ref{dai}). 
\begin{theorem}\label{sr} Let ${\small C_2 = {\sqrt{ 2}\,  \pi}  \frac{(1+ \omega)^{1/2}}{(\omega)^{1/2}} \, + 4 C_1 } $ where $\omega > 0$ and $C_1$  are respectively defined in Hypothesis \ref{d1}  and formula \eqref{e5}.
   For any   $s\in [0,1]$, $x,z\in H,$ $f \in B_b(H)$, $\lambda>0$, we have: 
\begin{equation}\label{dee}
|{P_s^{(z)}} D u^{(z)}(x) - D u^{(z)}(x) |_H  \le  C_2 \,  s^{1/2} \|f \|_{0}.  
\end{equation}
\end{theorem}
\begin{proof}
It is enough to prove the assertion when  $\lambda \in (0,1]$. Indeed once we have  proved
\begin{equation*}
|{P_s^{(z)}} D u^{(z)}(x) - D u^{(z)}(x) |_H  \le  C \,  s^{1/2} \|f \|_{0},\;\; \lambda \in (0,1],  
\end{equation*}
we write, for  $\lambda >1$,  using the previous resolvent identity  
\begin{gather*}
 u^{(z)}(x)  =  \int_0^{\infty} e^{-  t}  {P_{t}^{(z)} } f(x) dt  + {(1- \lambda )} \int_0^{\infty} e^{-   t}   {P_{t}^{(z)} }  u^{(z)} (x) dt.
 \end{gather*}
  It follows  that
 $|{P_s^{(z)}} D u^{(z)}(x) - D u^{(z)}(x) |_H   
 \le   C   s^{1/2} \big[\|f \|_{0} + \frac{ \lambda-1}{\lambda} \| f \|_0 \big]
   \le 2 C\,   s^{1/2} \|f \|_{0},  
$ 
since $ \| u^{(z)} \|_0 \le \frac{1}{\lambda} \|f \|_{0}$.  
 Now we  fix  $x,z\in H$, $s \in (0,1]$, $\lambda \in (0,1]$.   
  We have  
\begin{gather*} 
{P_s^{(z)}} D_h u^{(z)}(x) - D_h u^{(z)}(x) = \langle {P_s^{(z)}} D u^{(z)}(x) - D u^{(z)}(x), h \rangle  
\\ 
= \int_0^{\infty} e^{-\lambda t } [{P_s^{(z)}} (D_h {P_t^{(z)}} f)(x) -   D_{h} P_{t}^{(z)}f(x)]dt.  
\end{gather*}
Note that 
\begin{gather*}
  \int_0^{\infty} e^{-\lambda t }  D_{h} P_{t}^{(z)}f(x)dt = 
      \int_{0}^{s}e^{-\lambda t }  {    D_{h} P_{t}^{(z)}f(x)dt } \, + \,  e^{-\lambda 
 s}
\int_{0}^{\infty} e^{-\lambda t }  D_{h} P_{t+s}^{(z)}f(x)dt .
\end{gather*}
Hence 
\begin{gather*}
{P_s^{(z)}} D_h u^{(z)}(x) - D_h u^{(z)}(x) = \int_0^{\infty} e^{-\lambda t } {P_s^{(z)}} (D_h {P_t^{(z)}} f)(x) - e^{-\lambda s}
\int_{0}^{\infty} e^{-\lambda t }  D_{h} P_{t+s}^{(z)}f(x)dt 
\\ {   - }
 \int_{0}^{s} {   e^{-\lambda t } }   D_{h} P_{t}^{(z)}f(x)dt. 
   \end{gather*}
Since 
$\Big | \int_{0}^{s}e^{-\lambda u }   D_{h} P_{u}^{(z)}f(x)dt \Big| \le 2 C_1 \sqrt{s} |h|_H \| f\|_0, $   
   we concentrate on  $T_{x,s,f, \lambda}^{(z)} :  H \to \R$:
 \begin{equation}\label{w22}
T_{x,s,f, \lambda }^{(z)} \,( h)  = \int_0^{\infty} e^{-\lambda t } [{P_s^{(z)}} (D_h {P_t^{(z)}} f)(x) - e^{-\lambda s} D_{h} {P_{t+s}^{(z)}}f(x)]dt, \; h \in H.
\end{equation}
This linear    functional is well-defined because
\begin{gather*}
|{P_s^{(z)}} (D_h P_t^{(z)} f)(x) - e^{-\lambda s} D_{h} {P_{t+s}^{(z)}}f(x)|  
  \le \| D_h P_t^{(z)} f\|_0 + c(\lambda)\| D_{h} {P_{t+s}^{(z)}} f \|_{0} 
\le  \frac{c}{ \sqrt{t}} |h| \| f\|_0,
\end{gather*}
$ t>0, \, h \in H$ (using \eqref{e5}).  Moreover, it  is easy to check that it is linear and bounded (to this purpose, note that, for any $t>0$, the mapping $h \mapsto {P_s^{(z)}} (D_h P_t^{(z)} f)(x) -  e^{-\lambda s}D_{h} {{P_{t+s}^{(z)}}}f(x)$ is linear). 
   We will prove that 
\begin{equation}\label{2qq}
|T_{x,s,f, \lambda}^{(z)} (h) | \le  \frac{\pi} {\sqrt 2} \frac{(1+ \omega)^{1/2}}{(\omega)^{1/2}} \, s^{1/2} \, |h| \, \| f\|_0.
\end{equation}
 To this purpose let us consider $h = e_k$. 
Since in particular, for $t>0,$ $P_t^{(z)} f \in C^1_b(H)$   we can  differentiate under the integral sign and obtain  
\begin{gather*}
D_{e_k} {P_{t+s}^{(z)}} f(x) =
D_{e_k} {P_s^{(z)}} (P_t^{(z)} f)(x) = \int_H \langle D P_t^{(z)} f(e^{sA}x  + y + \Gamma_sz) , e^{sA} e_k \rangle \mu_s(dy)
\\
= e^{-\lambda _k s} {P_s^{(z)}} (D_{e_k} P_t^{(z)} f)(x), \;\; \; t>0.
\end{gather*}
Hence, for any $k \ge 1,$ $t>0,$
\begin{gather} \label{wqr}
{P_s^{(z)}} (D_{e_k} P_t^{(z)} f)(x) -  e^{-\lambda s} D_{e_k} {P_{t+s}^{(z)}}f(x) = [e^{\lambda_k s } -e^{-\lambda s}] D_{e_k} {P_{t+s}^{(z)}} f(x). 
\end{gather}
For  $h \in H$ we  define 
$
\pi_N h = h_N
$ $= \sum_{k=1}^N h^{(k)} e_k$ (cf. \eqref{pp1}). We have
\begin{gather*} 
 \begin{array}{l} \ds 
\int_0^{\infty} e^{- \lambda t} [{P_s^{(z)}} (D_{h_N} P_t^{(z)} f)(x) - e^{-\lambda s} D_{h_N} {P_{t+s}^{(z)}}f(x)] dt \\ \ds 
 = \sum_{k=1}^N h_k [e^{\lambda_k s } - e^{-\lambda s}] \int_0^{\infty} e^{- \lambda t}  D_{e_k} {P_{t+s}^{(z)}}f(x) dt.
 \end{array}  
\end{gather*}
 Let $S_N(s, \lambda ): H \to H$, 
 $S_N(s, \lambda) h $ $= \sum_{k=1}^N h_k (e^{s\lambda_k} -e^{-\lambda s}) e_k,$ $ h \in H.$
We have
\begin{gather*}
T_{x,s,f, \lambda}^{(z)} ( h_N)  = 
\sum_{k=1}^N h_k [e^{\lambda_k s } - e^{-\lambda s}] \! 
\! \int_0^{\infty} e^{- \lambda t}  D_{e_k} {P_{t+s}^{(z)}}f(x) dt
= \!\! \int_0^{\infty} e^{- \lambda t}  D_{S_N (s, \lambda) h} {P_{t+s}^{(z)}}f(x) dt.
\end{gather*}   
 Using Lemma \ref{ss1} we find
$ \ds |T_{x,s,f, \lambda}^{(z)} \,( h_N)|^2 
 \le \frac{\pi^2} { 2} \|f\|_0^2 \, \,  \| (-A)^{-1/2} e^{sA} S_N (s,\lambda ) \|_{\cal L}^2\,\, |h_N|^2_H.$

Since
 $\| (-A)^{-1/2} e^{sA} S_N (s, \lambda ) \|_{\cal L} $
$\le   \sup_{k \ge 1} [(1- e^{-[\lambda  + \lambda_k] s } )^{1/2}\,  \lambda_k^{-1/2}] \le  \,  s^{1/2} \sup_{k \ge 1} \frac{(1 + \lambda_k)^{1/2}}{(\lambda_k)^{1/2}},$
we get
 $|T_{x,s,f}^{(z)} \,( h_N)|^2 \le \frac{\pi^2} { 2} \frac{1+ \omega}{\omega}\,   \|f\|_0^2 |h_N|^2 \, s.$ 
 As $N \to \infty$ we get \eqref{2qq}.
 \end{proof} 
\def\ciao2{ 
 \begin{remark} {\em Estimate \eqref{dee} is not straightforward. It is equivalent to    
  \begin{gather*}
|{P_s^{(z)}}  D u^{(z)}(x) - D u^{(z)}(x ) |_H^2 =  \sum_{k \ge 1} |{P_s^{(z)}} D_{e_k} u^{(z)}(x) -  D_{e_k} u^{(z)}(x) |^2 \le C  s \|f \|_{0}^2 \;\; \text{where}
\\
 \sum_{k \ge 1} \Big | \int_0^{\infty} e^{-\lambda t } [{P_s^{(z)}} (D_{e_k} P_t^{(z)} f)(x) - D_{k} {P_{t+s}^{(z)}}f(x)]dt  \Big|^2 
\\  
=
\sum_{k \ge 1} \Big | \int_0^{\infty} e^{-\lambda t } \,  [e^{\lambda_k s } -1] \, 
|D_k {P_{t+s}^{(z)}} f(x)| dt  \Big|^2 .
\end{gather*}
 If instead of using Lemma \ref{ss1} we 
  put     the modulus inside the integral, we only get 
  \begin{gather} \label{222}
\int_0^{\infty} e^{-\lambda t } \,  |e^{\lambda_k s } -1| \, 
|D_k {P_{t+s}^{(z)}} f(x)| dt
\le C \| f\|_{0} \int_0^{\infty} \,  |e^{\lambda_k s } -1| \, \frac{ (\lambda_k)^{1/2}e^{-t\lambda_k} \, e^{-s\lambda_k} }
{
(1-e^{-2t\lambda_k})^{1/2}}  dt
\\ \notag
=  C \| f\|_{0} (1- e^{-\lambda_k s } ) \,\int_0^{\infty} \,   \, \frac{ (\lambda_k)^{1/2}e^{-t\lambda_k}}
{
(1-e^{-2t\lambda_k})^{1/2}}  dt
= 
  C \| f\|_{0} \frac{ (1- e^{-\lambda_k s } )} { \lambda_k^{1/2} }
\int_0^{\infty} \,   \, \frac{ e^{-u} \,  }
{
(1-e^{-2 u})^{1/2}}  du.
\end{gather}
 Now define
$
\phi(s) = \sum_{k \ge 1} \frac{ (1- e^{-\lambda_k s } )^2 }{\lambda_k}, \;\; s \in [0,1].
$

This  is a continuous function on $[0,1]$ such that $\phi(0)=0$. However   in general {\sl it is 
 not true that $\phi(s) \le C s$, $s \in [0,1]$} (one can consider the case $\lambda_k = k^2$). By the straightforward computations  in \eqref{222} one cannot get   \eqref{dee}.
 }
\end{remark}
}

\section{Proof of weak existence of Theorem \ref{base}}

We will prove weak existence  by adapting  a  compactness approach of  \cite{GG}.  This approach uses the factorization method introduced in  \cite{DKP} (this approach is also explained  in  Chapter 8 of \cite{DZ}).

Let us fix $x \in H$.  To construct the solution we start with some approximating mild solutions. We introduce, for each $m \ge 1$,  
$$
A_m = A \circ \pi_m ,\;\;\; A_m e_k =  - \lambda_k  \,  e_k,\;\; k = 1, \ldots m, 
$$
 $   A_m e_k =0, $ $k >m;$  here  
 $\pi_m=\sum_{j=1}^me_j\otimes e_j$ ($(e_j)$ is the basis of eigenvectors  of $A$; see \eqref{pp1}).

  For each $m$  there exists a weak mild solution $X_m= (X_m(t))_{t \ge 0}$ on some filtered probability space, possibly depending on $m$  (such solution can also be constructed by the Girsanov theorem, see \cite{GG0}, \cite{DZ} and  \cite{DFPR}). 
  
Usually the mild solutions $X^m$ are constructed on a time interval $[0,T]$.  However there is a standard procedure based on the Kolmogorov extension theorem to define the solutions  on $[0, \infty)$. On this respect, we refer to  Remark 3.7, page 303, in \cite{KS}.  

We know that  
 \begin{equation}\label{2ww}
X_m (t)=e^{tA} x +\int_{0}^{t}e^{\left(  t-s\right)  A} (-A_m)^{1/2}F^{}(X_m(s))
ds+\int_{0}^{t}e^{\left(  t-s\right)  A}dW_{s},\;\;\; t \ge 0.
\end{equation}
Recall    that, for any $t \ge 0,$ the stochastic convolution $W_A(t)= \int_{0}^{t}e^{\left(  t-s\right)  A}dW_{s}$ is a Gaussian random variabile with law $N(0,Q_t)$.  Let $p > 2$ and $q = \frac{p}{p-1} <2$.  We find (using also \eqref{ewd} {   and the H\"older inequality) }
\begin{gather*}
|X_m (t)|^p_H \le c_p ( |e^{tA} x|^p_H + \big | \int_{0}^{t}e^{\left(  t-s\right)  A} (-A_m)^{1/2}F^{}(X_m(s))ds |^p_H +
| W_A(t)|^p_H ) 
\\ \le 
  c_T| x|^p_H +  c_T ( \int_{0}^{t}  (t-s)^{-q/2} ds)^{p/q} \cdot    \int_{0}^{t} (1+ |X_m(s) |_H^p  )  ds  +
c_p | W_A(t)|^p_H 
 \\
\le C_T | x|^p_H + C_T +  C_T  \int_{0}^{t}   |X_m(s) |^p_H ds +
C_T | W_A(t)|^p_H,\;\;  t \in [0,T].  
\end{gather*}
 By the Gronwall lemma we find the bound
\begin{equation} \label{sd1}
\begin{array}{l}
\sup_{m \ge 1} \sup_{t \in [0,T]}\E|X_m (t)|^p_H = {   C_T } < \infty.
\end{array}  
\end{equation} 
{\sl The mild solution $X$ will be a weak limit of solutions $(X_m)$.} To this purpose we need some compactness results.
 The next  result is proved in \cite{GG} 
(the proof  uses that $(e^{tA})$ is a compact semigroup). 
\begin{proposition}
\label{p8.40}
If
$0<\frac{1}{p}<\alpha \leq 1$ then the
operator $G_{\alpha }: L^{p}(0,T;H) \to C([0,T];H) $
\begin{displaymath}
G_{\alpha }f(t)=\frac{\sin \pi \alpha}{\pi}  \int_{0}^{t}(t-s)^{\alpha -
1} e^{(t-s)A}f(s)ds,\quad f \in L^{p}(0,T;H),\;t \in [0,T], \;\; \text{is compact}. 
\end{displaymath}
 \end{proposition}
Below we consider a variant of the previous result. In the proof we use  estimate \eqref{ewd}.  
\begin{proposition}
\label{p8.41}
 Let $p>2$.   Then the
operator ${Q} : L^{p}(0,T;H) \to C([0,T];H)  $,
\begin{displaymath}
{Q} f(t)=\int_{0}^{t} (-A)^{1/2} e^{(t-s)A}f(s)ds,\quad f \in L^{p}(0,T;H),\;t \in [0,T], \;\; \text{is compact}. 
\end{displaymath}
\end{proposition}
\begin{proof} Since the proof is similar to the one of Proposition \ref{p8.40} we only give a sketch of the proof. 
  Denote by $| \cdot |_{p}$ the norm in
$L^{p}(0,T;H)$. According to the infinite 
dimensional version of the Ascoli-Arzel\'{a}
theorem one has to show that
\begin{enumerate}
\item[(i)] For arbitrary $t \in [0,T]$ the sets
$\{ {Q}f(t): \; |f|_{p}\leq 1\}$
 are relatively
compact in $H$.  

\item[(ii)] For arbitrary $\varepsilon >0$ there exists
$\delta >0$ such that
\begin{equation}
|{Q}f(t)-{Q}f(s)|_H \leq
\varepsilon,\;\; \text { if}\; \;  |f |_{p}\leq
1, \; |t-s|\leq \delta,\quad s,t \in [0,T].
\label{e8.17}
\end{equation}
\end{enumerate}
 To check (i)
let us fix $t \in (0,T]$ and define operators $Q^t $  and  $    Q^{\varepsilon,t }$ from  $L^{p}(0,T;H)$ into $H$, for
$\varepsilon  \in (0,t)$, 
\begin{displaymath}
{Q}^{t}f= Qf (t),\;\; \;\;
{Q}^{\varepsilon ,t}f=\int_{0}^{t-
\varepsilon } (-A)^{1/2} e^{(t-s)A} f(s)ds,\quad f \in
L^{p}(0,T;H).
\end{displaymath}
Since
 $
  {Q}^{\varepsilon ,t}f=e^{\varepsilon
A}\int_{0}^{t-\varepsilon } (-A)^{1/2} e^{(t-
\varepsilon -s) A}f(s)ds
$
and $e^{\varepsilon A},\varepsilon >0,$ is
compact, the  operators
${Q}^{\varepsilon ,t}$ are compact.
 Moreover, 
by using \eqref{ewd} and the 
H\"{o}lder inequality  (setting $q=\frac{p}{p-1} <2$)
\begin{displaymath}
 \begin{array}{l}
 \ds |{Q}^t f \, - \, {Q}^{\epsilon,t}f|_H =\left|\int_{t-\varepsilon }^{t } (-A)^{1/2} e^{(t-s)A}f(s)ds\right|_H 
\\ 
 \ds \leq M \left(\int_{t-\varepsilon }^{t }(t-
s)^{ - q/2}ds
\right)^{1/q}\left( \int_{t-\varepsilon }^{t} | f(s)|^{p}_H ds
\right)^{1/p} 
{   \le } \, \,  M_q \varepsilon ^{-1/2 \,  +1/q  } |f|_{p}
\end{array}
\end{displaymath}
with $ -\frac{1}{2} +  \frac{1}{q}  >0 $.  Hence
$Q^{\varepsilon,t  } \rightarrow
{{Q}^t} $, as $\epsilon \to 0^+$, in the operator norm so that
${Q}^t$ is compact and
(i) follows. Let us consider (ii).
For $0\leq t \leq t+u \leq T$ and $|f|_{p}\leq
1$, we have
\begin{displaymath}
 \begin{array}{l}
\ds |{Q}f(t+u)-{Q}f(t)|_H     
 \leq \int_{0}^{t} \|(-A)^{1/2} e^{(t+u-s)A} - 
(-A)^{1/2} e^{(t-s)A}\|_{\cal L}\,   |f(s)|_H ds 
 \\
\ds+\int_{t}^{t+u}| (-A)^{1/2}   e^{(t+u-s)A} f(s)|_H ds 
 \\
 \ds
 \le {   M_p} \Big(   \int_{0}^{u} s^{-q /2}ds
\Big)^{1/q}
+ 
 \Big( \int_{0}^{T} \|(-A)^{1/2} e^{(u+s)A}-  (-A)^{1/2}  e^{sA} \|^{q}_{\cal L} ds
\Big)^{1/q }   = I_1 + I_2.   
\end{array}
\end{displaymath}
It is clear that $I_{1} = M_p' \,   u^{1/2 -
1/p}      \rightarrow 0$ as
$u\rightarrow 0$.

     Moreover, for $s>0$, $(-A)^{1/2} e^{sA}$ is compact (indeed $(-A)^{1/2} e^{sA} e_k$ $= (\lambda_k)^{1/2} e^{-s \lambda_k} e_k$ and $(\lambda_k)^{1/2} e^{-s \lambda_k} \to 0$ as $k \to \infty$). It follows that  
  $\|e^{uA}(-A)^{1/2}e^{sA}-
(-A)^{1/2}e^{sA}\|_{\cal L}$ $ \rightarrow 0$ as $u\rightarrow 0$  for
arbitrary $s>0$.

Since
 $
\ds \|(-A)^{1/2} e^{(u+s)A}- (-A)^{1/2} e^{sA}
\|^{q}\leq \frac{ (2M)^q }{s^{q/2}} , \;   \,       s >0,
 $ $ u\geq 0, 
$
 and $q<2$,
 by the Lebesgue's dominated 
convergence theorem $I_{2}\rightarrow 0$ as   
$u\rightarrow 0$. Thus the proof of
(ii) is complete.
 \end{proof}

\noindent  {\bf Proof of the existence part of Theorem \ref{base}}. Let $x \in H$. We proceed {   in}  two steps. 
\\
{\it I Step.} Let $( X_{m})$ be solutions of \eqref{2ww}.  
 We prove that  their laws $\{ {\cal
L}(X_{m})\}$ form a tight family of 
probability measures on ${\cal B}(C([0,\infty);H))$.

To this purpose it is enough to show that for each $T>0$ the laws $\{ {\cal
L}(X_{m})\}$ form a tight family of 
probability measures on ${\cal B}(C([0,T];H))$.

 Let us fix  $p>2$, $T>0$ and choose $\alpha$ such that  $1/p <\alpha < 1/2 $. We know by \eqref{sd1}
 that there exists a constant $c_{p}>0$ such that $\E|X_{m}(t)|^{p}_H  $ $\leq c_{p},$ $  m\ge 1,$ $t \in [0,T].$ It follows that
\begin{equation}
\sup_{m \ge 1} \E\int_0^T |F_m(X_{m}(t))|^{p}_H < \infty.
\label{e8.18} 
\end{equation} 
with $\pi_m \circ F = F_m$, since   $|F_m(x)|_H \le C_F (1 + |x|_H)$, $m \ge 1$. 
 By the stochastic Fubini theorem we have the following
factorization formula (cf. Theorem 5.10 in \cite{DZ})
\begin{displaymath}
\begin{array}{l}
W_A(t)=\int_{0}^{t}e^{(t-s)A}dW_s=\; {G}_{\alpha} Y_t= G_{\alpha}(Y)(t), \quad t \in [0,T],
\end{array}
 \end{displaymath} 
where $
  Y_t=\int_{0}^{t} (t-r)^{-\alpha }e^{(t-r)A}   dW_r, $ $ t \in [0,T].
$ 
Therefore  
\begin{equation} \label{224c}
 \begin{array}{lll}
 X_{m}(t) 
&=& e^{tA} x+ Q(F_{m}(X_{m}))(t)+
   G_{\alpha
}(Y)(t),\;\; t \in [0,T].
\end{array}
\end{equation} 
Note that $Y_t$ is a Gaussian random variable with values in $H$, having mean 0 and  covariance operator $R_t = \int_0^t s^{-2 \alpha} e^{2sA} ds$. Therefore it is easy to prove that
\begin{equation} \label{s77} 
\E\int_{0}^{T}|Y_s|^{p}_H ds \le T \sup_{t \in [0,T]}  \E [|Y_t|^{p}_H] < \infty. 
\end{equation} 
 Now we  show tightness of $\mathcal L(X_{m})$ on  ${\cal B}(C([0,T];H))$.  
It follows from \eqref{e8.18},  
\eqref{s77} and  Chebishev's
inequality that for $\varepsilon >0$ one can   
find $r>0$ such that for all $m \ge 1$
\begin{equation}
\P\Big( \big(\int_{0}^{T}|Y_s|^{p}_H ds\big)^{1/p}   
\leq r \; \mbox{\rm and
} \big(\int_{0}^{T}|F_{m}( X_{m}(s) ) |^{p}_H ds \big)^  
{1/p}  \leq r \Big)> 1-\varepsilon.
\label{e8.19}   
\end{equation}
By Propositions \ref{p8.40} and   \ref{p8.41} (recall that $|\cdot|_p $ denotes the norm in
$L^{p}(0,T;H)$) the set
\begin{displaymath}
K=\{ e^{(\cdot )}x+ G_{\alpha}f(\cdot
)+ Qg(\cdot ): \; |f|_{p}\leq
r,|g|_{p}\leq r\} \subset C([0,T];H)  
\end{displaymath}
is  relatively  compact. It follows from \eqref{224c}
that 
 $\P (X_m \in K) = \mathcal  L(X_{m})(K)> 1-\varepsilon
,$ $ m \ge 1.$
and the tightness follows by the Prokhorov theorem.

\vskip 1mm       
\noindent 
  {\it II Step. } By the Skorokhod representation theorem, possibly passing to a subsequence of $(X_m)$ still denoted by $(X_m)$,  
 there exists a probability space $(\hat  \Omega, \hat  {\cal F}, \hat  \P )$ and  random variables $\hat X $ and  $\hat X_m$, $m \ge 1$, defined on $\hat \Omega $ with values in $C(0,\infty; H)$  such that the law of $X_m$ coincide with the law of $\hat X_m$, $m \ge 1,$ and moreover 
\begin{gather*}
 \hat X_m \to \hat X,\;\;\; \hat \P-a.s. 
\end{gather*}
Let us fix ${k_0} \ge 1$. Let $\hat X^{(k_0)}_m = \langle \hat X^{}_m,  e_{k_0} \rangle
$.  Recall that $\pi_m \circ F = F_m$.
 It is not difficult to prove that     the processes $(M_{m}^{(k_0)})_{m \ge 1}$  
$$
M_{m}^{(k_0)}(t) = \begin{cases} \hat X^{(k_0)}_m(t) - x^{(k_0)} +   \lambda_{k_0} \int_0^t  \hat X^{(k_0)}_m(s) ds \, - \, \lambda_{k_0}^{1/2} \, \int_0^t F^{(k_0)}(\hat X_m (s))ds,\; \; {k_0} \le m 
\\     \hat X^{(k_0)}_m(t) - x^{(k_0)} +      \lambda_{k_0} \int_0^t \hat X_m^{(k_0)}(s) ds,    \, \, \; \;\; {k_0} > m,\;\; t \ge 0,   
\end{cases}
$$
are square-integrable continuous  ${\cal F}_t^{\hat X_m}$-martingales on $(\hat  \Omega, \hat  {\cal F}, \hat  \P )$ with $M_{m}^{(k_0)}(0) =0$.
 Moreover the quadratic variation process  $\langle M_{m}^{(k_0)}\rangle_t \, = \, t$, $m \ge 1$ (cf. Section 8.4 in \cite{DZ}).  
 
 Passing to the limit as $m \to \infty$ we find that 
\begin{equation} \label{w11}
M_{}^{(k_0)}(t) =  \hat X^{(k_0)}(t) - x^{(k_0)} +    \lambda_{k_0} \int_0^t  \hat X^{(k_0)}(s) ds \, - \, \lambda_{k_0}^{1/2} \, \int_0^t F^{(k_0)}(\hat X (s))ds, \;  t \ge 0,   
   \end{equation}
 is a square-integrable continuous  ${\cal F}_t^{\hat X}$-martingale with $M_{}^{(k_0)}(0) =0$. To check the martingale property,    
 let us fix $0 < s < t $. We know that $\hat \E [M_{m}^{(k_0)}(t) - M_{m}^{(k_0)}(s) /\,  {\cal F}_s^{\hat X_m}]=0$, $m \ge 1$.   
  
  Consider $0 \le s_1 < \ldots < s_n \le  s$, $n \ge 1$. For  any   $h_j \in C_b (H)$, we have, for $m \ge k_0$,
\begin{gather}\label{w113}  
 \begin{array}{l}
 \hat \E  \Big[ \big ( \hat X^{(k_0)}_m (t) - \hat X^{(k_0)}_m(s) +    \lambda_{k_0} \int_s^t  \hat X_m^{(k_0)}(r) dr  -  \lambda_{k_0}^{1/2} \, \int_s^t F^{(k_0)}(\hat X_m (r))dr    \big )  \\
    \cdot \prod_{j=1}^n 
h_j(\hat X_m ({s_j}))\Big]=0.
\end{array}     
\end{gather}
 Using that $| F^{(k_0)} (x) | \le C_F   (1 + |x|_H)$ and that,  for any  $T>0$,
 $\ds \sup_{m \ge 1} \sup_{t \le T} \hat \E [ |\hat X_m(t)|^p_H ]
 \le C < \infty$
 (cf. \eqref{sd1}) 
by the Vitali   convergence theorem we get easily that \eqref{w113} holds when
 $\hat X_m$ is replaced by $\hat X$ {   (note that  this assertion could   be proved by using only the   dominated convergence theorem).} Then  we obtain that $M_{}^{(k_0)}$ 
 is   a square-integrable continuous  ${\cal F}_t^{\hat X}$-martingale.   
  \\ 
Moreover, by a limiting procedure, arguing as before, we find that $( (M_{}^{(k_0)}(t))^2 -t)$ is a martingale. It follows that $M_{}^{(k_0)}$ is {\it a real Wiener process} on $(\hat \Omega, \hat {\cal F}, \hat \P)$.

Hence, for any $k \ge 1,$ we find that there exists a real Wiener process $M^{(k)}$ such that
\begin{gather*}
\begin{array}{l}
  \hat X^{(k)}(t) =  x^{(k)} -    \lambda_{k} \int_0^t  \hat X^{(k)}(s) ds 
  \, + \, \lambda_{k}^{1/2} \, \int_0^t F^{(k)}(\hat X (s))ds + M_{}^{(k)}(t). 
  \end{array} 
\end{gather*}
We prove now that   $(M^{(k)})_{k \ge 1}$ are independent Wiener processes. 
\\
We fix $N \ge 2$ and
 introduce the processes $(S_{m}^N)_{m \ge 1}$, $S_m^{N}(t) =  
 \big ( M_{m}^{(1)}(t), \ldots,   \ M_{m}^{(N)}(t)\big)_{t \ge 0}$, with values in $\R^N$.     The components of $S_m^{N}$  are square-integrable continuous  ${\cal F}_t^{\hat X_m}$-martingales. Moreover  the    quadratic covariation $\langle M_{m}^{(i)},  M_{m}^{(j)} \rangle_t = \delta_{ij} t $.
 
 Passing to the limit as before we obtain that also the $\R^N$-valued process $(S^N(t))$, 
 $
 S^N (t) = \big ( M^{(1)}(t), \ldots,   \ M^{(N)}(t) \big), \;\; t \ge 0,  
  $
 has  components which are square-integrable continuous  ${\cal F}_t^{\hat X}$-martingales with quadratic covariation $\langle M^{(i)},  M^{(j)} \rangle_t = \delta_{ij} t $. Note that $S^N(0)=0$, $\hat \P$-a.s.
 
  By the L\'evy characterization of the Brownian motion (see Theorem 3.16 in \cite{KS}) we have that  $\big ( M^{(1)}(t), \ldots,   \ M^{(N)}(t) \big)$ is a standard Wiener process with values in $\R^N$. Since $N$ is arbitrary, 
  $(M^{(k)})_{k \ge 1}$ are independent real Wiener processes and  the proof is complete.

\begin {remark} \label{sd} {\em 
Following the previous method one can prove existence of weak mild solution even for 
\begin{equation*} 
dX_{t}=AX_{t}dt +  (-A)^{\gamma}F(X_{t})dt+ dW_{t},\qquad X_{0}=x\in H,
\end{equation*}
with $\gamma \in (0,1)$ and $F: H \to H$  continuous and having at most a linear growth.   
}    
\end {remark}

\section {Proof of weak uniqueness  when $F \in C_b(H,H)$}

To get the weak uniqueness of Theorem \ref{base} when 
 $F \in C_b(H,H)$
we  first show  the equivalence between martingale solutions and mild solutions. Indeed  for martingale problems  some useful uniqueness  results  are  available even in infinite dimensions (see, in particular, Theorems \ref{ria}, \ref{uni1} and \ref{key}).

\subsection{Mild solutions and  martingale problem   } 

 We  formulate the martingale problem of Stroock and Varadhan \cite{SV79} for the operator $\L$ given below in \eqref{ll} and associated to \eqref{sde}. We stress that  an infinite-dimensional generalization of the martingale problem is proposed in Chapter 4 of \cite{EK}.  Here we follow 
   Appendix of \cite{PrPot}.  In such  appendix  some extensions and modifications of theorems given in
Sections 4.5 and 4.6 of \cite{EK} are proved.

 \vskip 1mm       
   We use the space $C^2_{cil}(H)$
  of regular
  cylindrical functions (cf. \eqref{cil2}).
  We 
 deal with the following linear operator ${\L}: D(\L) \subset C_b(H) \to C_b(H)$,
with $D(\L) = C^2_{cil}(H)$  (recall that here   $F \in C_b(H,H)$): 
 \begin{eqnarray} \label{ll}
\L f (x) &=& \frac{1}{2} Tr(D^2 f(x)) + \langle x, ADf(x) \rangle +
\langle F(x),  (-A)^{1/2} Df(x) \rangle 
\\ \nonumber & = &
L f (x) +  
\langle F(x),  (-A)^{1/2} Df(x) \rangle,\;\;\; f \in D(\L),\; x \in H.
\end{eqnarray}

\begin{remark} {\em   
 We stress  that the linear operator $(\L, D(\L ))$  in \eqref{ll} is countably pointwise determined,
   i.e., it verifies Hypothesis 17 in \cite{PrPot}. Indeed, arguing  as in Remark 8 of \cite{PrPot}, one shows that there exists  a countable
set ${\mathcal H}_0 \subset D(\L)$ such
 that for any $f \in D(\L) $, there exists 
a sequence $(f_n) \subset {\mathcal H}_0$ satisfying
$$
 \lim_{n \to \infty} ( \| f - f_n \|_{0} +   \|  \L f_n -   \L f \|_{0}) =0. \qed
$$ 
}
\end{remark}

Let $x \in H$. An $H$-valued stochastic process $X = (X_t)= (X_t)_{  t \ge 0 }$
defined on some probability space $(\Omega, {\cal F}, \P)$
with continuous trajectories is a \textsl { solution of the martingale
problem for $(\L, \delta_x)$} if,
 for any $f \in D(\L)$,
\begin{equation}\label{mart}
 \begin{array}{l}
    M_t(f) = f(X_t) - \int_0^t \L f(X_s) ds, \;\; t \ge 0, \;\; \text{is a
    martingale}
\end{array} 
    \end{equation}
(with respect to the natural filtration $({\cal F}_t^X)$) 
  and, moreover,     $X_0 =x, \P$-a.s.. 
  
  If we do not assume that $F$ is bounded then in general    
   $M_t(f)$ is only a local martingale because   in general  $\L f$ is not    a bounded function.

 \smallskip 
  We say that \textsl{ the martigale problem for ${ \L}$ is  well-posed} if, for any $x  \in H$,   there exists a martingale solution
 for $({ \L}, \delta_x)$ and, moreover,  uniqueness in law    holds  for  the martingale problem for $({ \L}, \delta_x)$.

  Equivalence between 
    mild solutions and martingale solutions  has been proved in a  general setting  in  \cite{kunze} even for SPDEs in Banach spaces.   
  We only give a sketch of the proof of  the next result for the sake of completeness
  (see also Chapter 8 in \cite{DZ}). 
  

\begin{proposition} \label{ser}  Let   $F \in C_b(H,H)$ and $x \in H.$   

(i)  If $X$ is a weak mild solution to \eqref{sde}    with $X_0 = x$, $\P$-a.s.,  then
 $X$ is also 
a  solution  of the martingale
problem for $(\L, \delta_x)$.

(ii) Viceversa, if $X= (X_t)$ is  a  solution of the martingale
problem for $(\L,  \delta_x)$ on some probability space $(\Omega, {\cal F}, \P)$ then there exists a cylindrical Wiener process on $(\Omega, {\cal F},({\cal F}_t^X), \P)$ such that 
$X$ is  a weak mild solution to \eqref{sde}  
  on $(\Omega, {\cal F},({\cal F}_t^X), \P)$
with initial
condition $x$. 
\end{proposition}
\begin{proof}(i) Let   $X$ be  a weak mild solution to \eqref{sde} with $X_0 = x$, $\P$-a.s.
 defined on  a filtered probability space $(
\Omega,$ $ {\mathcal F},
 ({\mathcal F}_{t}), \P) $. Let $f \in D(\L)$. Since  $f$ depends only on a finite number of variables  by the It\^o formula   
  we obtain that
$
f(X_t) $ $ - \int_0^t \L f(X_s) ds
$
 is an ${\cal F}_t$-martingale, for any $f \in D(\L)$. 
 We get easily  the assertion since ${\cal F}_t^X \subset 
{\cal F}_t$, $t \ge 0$.

\vskip 1mm \noindent  (ii)  Let $X$ be a solution to the martingale problem
 for $(\L, \delta_x)$ defined on $(\Omega, {\cal F}, \P) $.

\vskip 0.5 mm       
\noindent  \textit{I Step.}  {\it Let $X^{(k)}_t = \langle X_t, e_k \rangle $ and $F(x) = \sum_{k \ge 1} F^{(k)}(x) e_k$.
We show that, for any $k \ge 1$,
$$
 \begin{array}{l}
X^{(k)}_t - x^{(k)} + \lambda_k\int_0^t X^{(k)}_s ds - \int_0^t (\lambda_k)^{1/2}F^{(k)}(X_s) ds
\end{array} 
$$
is a one-dimensional Wiener process $W^{(k)}= (W^{(k)}_t)$.}
 
Let $k \ge 1$. We will  modify   a well known argument
 (see, for instance, the proof of Proposition 5.3.1 in \cite{EK}).
 By the definition of martingale solution,
it follows easily that if
$f(x)= x^{(k)} = \langle x, e_k\rangle $, $x \in H$, the process
 \begin{equation} \label{f8}
\begin{array}{l}
M_t^{(k)} = 
 X^{(k)}_t  -  x^{(k)} + \lambda_k\int_0^t X^{(k)}_s ds - \int_0^t b_k(s) ds \,\; \text{is a continuous  local martingale},
 \end{array} 
 \end{equation}
   which is ${\cal F}_t^X$-adapted,  with    
 $b_k(s) = (\lambda_k)^{1/2}F^{(k)}(X_s)$ {  (to this purpose one has  to approximate the unbounded function $l_k(x)=\langle x, e_k\rangle$ by functions 
  $l_k (x) \eta (\frac{\langle x, e_k\rangle}{n})$, $n \ge 1$, where  
  $\eta \in C_0^{\infty}(\R)$ is such that  $\eta(s)=1$ for $|s| \le 1$).
 }     Then
    using $f(x) = (\langle x, e_k\rangle)^2 $, $x \in H$, we find that
\begin{equation} \label{g5}
\begin{array}{l}     
N_t^k = (X^{(k)}_t)^2 - (x^{(k)})^2 + 2\lambda_k\int_0^t (X^{(k)}_s)^2 ds -
2\int_0^t b_k(s)\, X^{(k)}_s  ds - t,
\end{array} 
\end{equation}
is also a continuous local martingale.  On the other hand, starting from \eqref{f8}   and
applying the It\^o formula (cf. Theorem 5.2.9 in \cite{EK}), we get
$$
\begin{array}{l}
(X^{(k)}_t)^2 = (x^{(k)})^2 - 2\lambda_k\int_0^t (X^{(k)}_s)^2 ds +
2\int_0^t b_k(s) \,X^{(k)}_s  ds + 2 \int_0^t b_k(s)   dM_s^{(k)} +  \langle M^{(k)} \rangle_t,
\end{array} 
$$
where $(\langle M^{(k)} \rangle_t)$      is the variation process of $M^{(k)}$.
Comparing this identity with \eqref{g5} we deduce: $N_t^k - 2 \int_0^t b_k(s)   dM_s^{(k)}= \langle M^{(k)}
\rangle_t -t $ and so 
 $\langle M^{(k)}
\rangle_t =t$  
 (a continuous local martingale of bounded variation is constant).
 By the  L\'evy martingale characterization of the Wiener process (see Theorem 5.2.12 in \cite{EK}) we get that $M^{(k)}$ is a real Wiener process.

\vskip 0.5 mm       
 \noindent \textit{II Step.} \textsl{   We prove that the previous Wiener processes $W^{(k)} = M^{(k)} $   are independent.}

\vskip 0.5 mm  We fix any   $N \ge 2$ and prove that $W^{(k)}$, $k=1, \ldots, N$ are independent. 
  We will argue similarly  to the first step.
By using  functions $f(x)= x^j x^k$, $j,k \in \{ 1, \ldots N\}$, we get that
$
\langle W^{(j)}, W^{(k)} \rangle_t
$ $ = \delta_{jk} t.
$
 Again by the  L\'evy martingale characterization of the Wiener process (cf. Theorem 3.16 in \cite{KS})
we get that $(W^{(1)}, \ldots, W^{(N)})$ is an $N$-dimensional standard Wiener process.
It follows that $\{W^{(k)} \}_{k =1, \ldots, N}$ are independent real  Wiener processes.
\end{proof}
For the martingale problem for $\L$ in \eqref{ll}  we have  the following  uniqueness  result (we refer to   Corollary 21 in \cite{PrPot}; see also Theorem 4.4.6 in \cite{EK} and Theorem 2.2 in \cite{kunze}).

\begin{theorem} \label{ria} 
  Suppose the following two conditions:

 (i)  for any $x \in H$, there exists a  martingale
solution
  for $(\L, \delta_x)$;

 (ii)  for any $x \in H$, any two  martingale solutions $X$ and $Y$ for $(\L, \delta_x)$ have the same one dimensional marginal laws (i.e., for    $t \ge 0$, the law of $X_t$ is the same as $Y_t$ on ${\cal B}(H)$).   

 Then the
   martingale
   problem for  $\L$ is well-posed. 
\end{theorem}
Throughout  Section 5 we will apply the previous result and also the next localization principle for $\L$ (cf.  Theorem 26 in \cite{PrPot}). 
\begin{theorem} \label{uni1}  Suppose   
    that for any $x \in
H$ there exists a martingale solution  for  $(\L, \delta_x)$.
   Suppose that
 there exists
a family $\{ U_j\}_{j \in J}$ of open sets $U_j \subset H$  with $\cup_{j \in J} U_j = H$  and  linear operators $\L_j$ with the same domain of $\L$, i.e., 
 $ \L_j: D(\L) \subset C_b (H) \to C_b(H)$, $j \in J $ such that

i)  for any $j \in J$, the martingale problem for $\L_j$ is well-posed.

ii) for any $j \in J$, $f \in D(\L)$, we have
$
 \L_j  f(x) = \L  f(x),\;\; x \in U_j.
$
\\ 
Then the martingale problem for $\L$ is well-posed.  
  \end{theorem}
  
 In {  Sections 6 and 7}  we   treat {\sl possibly unbounded} $F$; we will prove uniqueness 
by truncating $F$ 
 and using uniqueness  for the martingale problem up to a stopping time.   According to  Section 4.6 of \cite{EK} this leads to the  concept of {\sl stopped martigale problem} for  $\cal L$ which we define now.  
  
\smallskip   
Let us fix an open set $U \subset H$ and consider the Kolmogorov operator $\L$ in \eqref{ll} with $F \in C_b(H,H)$. 

 Let $x \in H$.  A  stochastic process $Y = (Y_t)_{t \ge 0}$ with values in $H$
defined on some probability space  $ (\Omega, {\cal F}, \P)$
with continuous paths 
is a  solution of the stopped martingale problem for $(\L, \delta_x, U)$ if  $Y_0 =x$, $\P$-a.s.  and the
 following conditions hold:

 (i) $Y_t = Y_{t \wedge \tau}$, $t \ge 0,$  $\P$-a.s, where
\begin{equation} \label{ta1}
 \begin{array}{l}
\tau = \tau^Y_U= \inf \{ t \ge 0 \; : \; Y_t \not \in U  \}
 \end{array}
 \end{equation}   
($\tau = + \infty$ if the set is empty;  this exit time $\tau$ is an
${\cal F}_t^Y$-stopping time);

 (ii) for any $f \in D(\L) =  C^2_{cil}(H)$, 
  $  \ds  f(Y_t) - \int_0^{t \wedge \tau}
    \L f(Y_s) ds, \;\;\; t \ge 0, \;\; \text{is a ${\cal F}_t^Y$-martingale.}$
 
 A key result
 says, roughly speaking,  
 that if the (global) martingale problem for an operator  is well-posed then also the
 stopped martingale problem for such operator is well-posed for any choice
  of the open set $U$ and for any initial condition $x$ (we refer to  Theorem 22 in \cite{PrPot}; see   also the beginning of Section A.3 for a comparison  between this result and  Theorem 4.6.1 in \cite{EK}).   
   We state this  result for the operator $\L$ in \eqref{ll}.
   \begin{theorem} \label{key} Suppose  that
 the  martingale problem for $\L$ is well-posed.

  Then also the
 stopped martingale problem for $(\L, \delta_x, U)$ is well-posed for any  $x \in H$ and for any open set $U$ of $H$. In particular uniqueness in law holds for the
 stopped martingale problem  for $(\L, \delta_x, U)$, for any $x \in H$ and   $U$ open set in $H$.  
 \end{theorem}

In order to apply Theorems \ref{ria} and \ref{uni1} we need 
 existence  of regular solutions for   related Kolmogorov equations
 and  some convergence results. This will be done in the next section.

\subsection{On the Kolmogorov
equation for $\L$ when $\| F- z \|_0 < 1/4 $   }

Here we study  the Kolmogorov equation   
\begin{equation}
\label{e12} \lambda u-L u-\langle (-A)^{1/2}Du, F   \rangle=f,
\end{equation}
where $\lambda>0$, $f\in C_b^2 (H)$ and $F\in C_b(H,H)$ (cf. \eqref{ll};  $L$ is  the Ornstein-Uhlenbeck operator).  
 We assume that  there exists $z \in H$ such that 
 \begin{equation}\label{111}
 \sup_{x \in H}|F(x) - z|_H < 1/4. 
 \end{equation}
We will prove  regularity and convergence   results  for   solutions.
 Note that to study \eqref{e12} we cannot proceed as in Proposition 5 of \cite{DFPR}   because in general $\| (-A)^{1/2} Du \|_0 \not \to 0$ 
 as $\lambda \to \infty$.  
 
 We will rewrite the equation as 
\begin{equation}
\label{eee} \lambda u(x)-L u(x)-\langle (-A)^{1/2}Du(x), z   \rangle=f(x) + \langle (-A)^{1/2}Du(x),   F(x)  -z \rangle.
\end{equation}
 Let us introduce the Banach space $E = \{ v \in C^1_b(H), \;\; Dv(x)  \in D((-A)^{1/2}),\; x \in H , \;\; (-A)^{1/2} Dv \in B_b(H,H)\}$  endowed with the norm
$$ 
\| v\|_E = \| v\|_0 +  \|  (-A)^{1/2} Dv\|_0, \;\;\; v \in E.
$$
We first prove  
\begin{lemma} \label{stima1} 
 For any $z \in H$,  $F \in C_b(H,H)$ which  verify 
$\| F- z \|_0 < 1/4,$
for any $\lambda  \ge 1 $, $g \in B_b(H)$,  there exists  a unique solution $u= u^{(z)}$ in $E$ to the integral equation
\begin{gather*}
u^{(z)} (x) =  \int_0^{\infty} e^{-\lambda t }  P_{t}^{(z  )} [g + \langle   [  F   - z]  , (-A)^{1/2}  Du^{(z)}    \rangle ] (x)dt,\;\;\; x \in H
\end{gather*}
 (we drop the dependence of $u^{(z)}$ from $\lambda$).
 Moreover   
\begin{gather*}
\| (-A)^{1/2}  Du^{(z)} \|_0 \le  12 \| g\|_0,\;\;\;  \| u^{(z)}\|_0 \le  4  \| g\|_0.
\end{gather*}
\end{lemma} 
 \begin{proof}
  We define $T: E \to E$, $Tu(x)=  \int_0^{\infty} e^{-\lambda t }   P_{t}^{( z  )} [g + \langle   [   F   - z]  , (-A)^{1/2}  Du    \rangle ] (x)dt,$ $u \in E$, $ x \in H$. Note that if $u \in E$ then $g + \langle   [  F   - z]  , (-A)^{1/2}  Du    \rangle  $ $\in B_b(H)$ and so by {    Theorem \ref{ss13}, } $Tu \in E$.
    We prove that $T$ is a strict contraction. Since 
 \begin{gather*}
Tu(x) - Tv(x) =  \int_0^{\infty} e^{-\lambda t }   P_{t}^{ (z  )} (\langle   [  F   - z]  , (-A)^{1/2}  [Du - Dv]    \rangle )  (x)dt, 
\end{gather*}
 we find 
$\| Tu - Tv  \|_0 \le \frac{1}{ 4 \lambda} \|u-v \|_E \le  \frac{1}{4} \|u-v \|_E$
  and by Theorem \ref{ss13}, {   since $\frac{\pi} {\sqrt 2} < 3$, }
  \begin{gather*}
\| (-A)^{1/2}  D [Tu]  - (-A)^{1/2}  D [Tv]   \|_0 \le \frac{3}{4}  \|u-v \|_E. 
\end{gather*}
  Hence we have a unique fixed point $u^{(z)}\in E$ which solves the integral equation.   Moreover
\begin{gather*}
\| (-A)^{1/2}  Du^{(z)} \|_0 \le  3 (\| g\|_0 + \frac{1}{4} \| (-A)^{1/2}  Du \|_0)
\end{gather*}
(see Theorem \ref{ss13}) and the assertion follows.
  \end{proof}

Let  $F \in C_b(H,H)$ which    {   verifies}  \eqref{111}.  
 Set, for any $n \ge 1,$ 
\begin{equation}
\label{e29}
 \begin{array}{l}
F_n(x)=\int_H F(e^{\frac1nA}x+y)N(0,
Q_{\frac1n})(dy),\quad x\in H.
\end{array} 
\end{equation}
Then $F_n$ is of $C^\infty$ class and all its derivatives are
bounded. Moreover $\|F_n\|_0\le \|F\|_0$, $n \ge 1$. It is not difficult  to prove  that 
 \begin{equation}
\label{e30} F_n(x)\to F(x),\;\; x \in H, 
\end{equation}
as $n \to \infty$,  and $\| F_n- z \|_0 < 1/4 $, for any $n \ge 1$. 

Recalling  that in \eqref{e12} $f \in C^2_b(H)$ we  consider   classical bounded solutions to    the following finite-dimensional  equations
\begin{equation} 
\label{e3351}
\lambda u _{nm}- L u _{nm}-\langle  (-A)^{1/2} \pi_m F_n \circ \pi_m, Du _{nm}   \rangle= f  \circ \pi_m, \;\; n,m \ge 1,\; \lambda \ge 1.
\end{equation}
where  $\pi_m=\sum_{j=1}^me_j\otimes e_j$.  We write $ f_m = f  \circ \pi_m$, $z_m = \pi_m z $ and
 $F_{nm} = \pi_m F_n  \circ \pi_m$, $A_m =  A \, \pi_m $.  Note that  
 \begin{equation} \label{w55} 
  \begin{array}{l}
  \| F_{nm} -  z_m \|_0 < 1/4,\;\;\; n,m \ge 1.
\end{array}  
  \end{equation}
 We have  $u_{nm} = u_{nm} \circ \pi_m$, 
  $L  u _{nm} = L_m  u _{nm}$  with   $ L_{m} u _{nm} $ $=\frac12\;\mbox{\rm Tr}\;[ D^2 u _{nm} (x)]$ $+\langle A_m x,D u _{nm} (x) \rangle $, $x\in H.$  
  Indeed, for  $\lambda  \ge 1$, $n,m \ge 1$,  equation \eqref{e3351} can be solved by  considering the  associated  equation in $\R^m$ which is like
\begin{equation} \label{w33}
\lambda v(y)- \frac{1}{2} \triangle v(y) - \langle  B y  , Dv(y)    \rangle -  \langle 
G(y) , Dv(y)    \rangle= g(y),\;\;\; y \in \R^m,
\end{equation}
where $B$ is a given $m \times m$ real matrix and $g,G$ are regular and bounded functions 
 (to this purpose one can use, for instance,  the Schauder estimates proved in \cite{DL}).

Thus, for any $m,n \ge 1$, there exist classical cylindrical functions $u_{nm} \in C^2_b (H)$ which solve 
\eqref{e3351}. Such functions are the unique bounded classical solutions; however in order to prove uniqueness for SPDE \eqref{sde} it is  important  to show existence of classical solutions.
    We can rewrite  \eqref{e3351} as 
 \begin{gather*}
\lambda u _{nm}(x)- L u _{nm}(x)-\langle  (-A)^{1/2}  z_m , Du _{nm} (x)  \rangle
\\    = f  \circ \pi_m(x) + \langle  (-A )^{1/2}  [  \pi_m F_n  \circ \pi_m (x) - z_m]  , Du _{nm}(x)   \rangle,\;\;\; x \in H,
\end{gather*}
and so we obtain the following finite-dimensional representation formula:
 \begin{equation}\label{344}  
u_{nm} (x) =  \int_0^{\infty} e^{-\lambda t }  {    P_{t}^{(z_m)} } [f_m 
+ \langle   [  \pi_m F_n  \circ \pi_m - z_m]  , (-A)^{1/2}  Du _{nm}   \rangle ] (x)dt,\;\;\; x \in H. 
\end{equation}  
Note that, since $f_m = f \circ \pi_m$,    
\begin{gather*}
  P_{t}^{( z)} f_m  (x) =
 P_{t}^{( z_m)} f_m  (x)   =\, \int_{H} f(e^{t A_m } x + \pi_m y + (-A_m)^{-1/2}[z_m - e^{tA_m} z_m] )  \; { N} \big (0 ,   Q_t )\, (dy).
\end{gather*}
By  Lemma \ref{stima1}  we have the bounds   
 \begin{gather} \label{2ee}
 \|  (-A)^{1/2}  Du _{nm} \|_0 \le 12 \| f\|_0,\;\;\;\; 
 \|    u _{nm} \|_0 \le 4 \| f\|_0 .
\end{gather}
  Now let us introduce, for $x \in H,$ $m \ge 1$, $\lambda \ge 1$, the solution $u_m = u_m^{(z)} \in E$ to 
 \begin{gather} \label{sqq}
 u_{m} (x) =  \int_0^{\infty} e^{-\lambda t }   P_{t}^{( z  )} [f_m + \langle   [  \pi_m F   \circ \pi_m - z_m]  , (-A)^{1/2}  Du _{m}   \rangle ] (x)dt,\; x \in H. 
 \end{gather}
 By Lemma \ref{stima1}, we know that  
 \begin{equation}\label{e66}
\| (-A)^{1/2}  Du_m \|_0 \le  12 \| f\|_0.
\end{equation}
  \begin{lemma} \label{stima2} Let $\lambda \ge 1$,  $z \in H$, $f \in C^2_b (H)$  and  consider  classical bounded solutions $u_{nm}$ of equation \eqref{e3351} when 
  $F \in C_b(H,H)$ verifies \eqref{111}.
   We have, for any $x \in H$, $m \ge 1$, 
\begin{align} \label{conv1}
\lim_{n\to \infty}u_{nm}(x) =u_m(x) \;\; \text{and} 
\;\;\; {    \sup_{n,m \ge 1}  } \| u_{nm}  \|_{0} =
C_{} < \infty,
 \end{align}
\end{lemma}
\begin{proof} We only need to  prove the first assertion. Let us fix $m \ge 1$.
\begin{gather*}
u_{nm}(x) - u_m(x)
 = \int_0^{\infty} e^{-\lambda t }   P_{t}^{( z )} \Big (\langle   [  \pi_m F_n   \circ \pi_m  -    \pi_m F   \circ \pi_m  ]
, (-A)^{1/2}  Du _{nm}   \rangle
\\
+ \langle   [  \pi_m F   \circ \pi_m - z_m]  , (-A)^{1/2}  Du _{nm} - (-A)^{1/2}  Du _{m}   \rangle \Big)(x)dt.
\end{gather*}
 Using the uniform bound on  $\| (-A)^{1/2}  Du _{nm} \|_0$ and the fact that 
    $(\pi_m F_n   \circ \pi_m  -    \pi_m F   \circ \pi_m  )$ is uniformly bounded and converges pointwise to zero as $n \to \infty$ we obtain  
\begin{gather*}
\big | \int_0^{\infty} e^{-\lambda t }   P_{t}^{(z )} \big( \langle   [  \pi_m F_n   \circ \pi_m  -    \pi_m F   \circ \pi_m  ]
, (-A)^{1/2}  Du _{nm}   \rangle \big) (x) \big | dt \to 0,\;\; \text{as}\, n \to \infty.
\end{gather*}    
 It remains to prove that
 \begin{gather*}
\lim_{n \to \infty}\big | \int_0^{\infty} e^{-\lambda t }   P_{t}^{(z )}  \big(
 \langle   [  \pi_m F   \circ \pi_m - z_m]  , 
 (-A)^{1/2}  Du _{nm} - (-A)^{1/2}  Du _{m}   \rangle \big) (x)dt \big | =0. 
\end{gather*}
Using the bound  \eqref{e66} and Theorem \ref{ss13} (see also  \eqref{w12})    we can apply the dominated convergence theorem and obtain the assertion.    
\end{proof}

\begin{lemma} \label{stima3} Let $\lambda \ge 1$, $z \in H$,  $f \in C^2_b (H)$  and  consider  $u_{m}$ given in \eqref{sqq} with $F$ verifying \eqref{111}.  We have, for any $x \in H$, 
\begin{align} \label{conv13}
\lim_{m\to \infty}u_{m}(x) =u(x) \;\; \text{and}  
\;\;\; \sup_{m \ge 1} \| u_{m}  \|_{0} < \infty,
 \end{align} 
\end{lemma}
\begin{proof} The second bound is clear by Lemma \ref{stima1}.
 We  prove the first assertion.  
\begin{gather*}  
\begin{array}{l}
 u_{m}(x) - u(x)
  = \int_0^{\infty} e^{-\lambda t }   P_{t}^{(z)}  \big (f_m - f
 + \langle     (\pi_m F   \circ \pi_m  - z_m) -    (  F -z)    
, (-A)^{1/2}  Du _{m}   \rangle \big )(x) dt 
\\
+ \int_0^{\infty} e^{-\lambda t }   P_{t}^{(z)}  \big (   \langle      F    - z  , (-A)^{1/2}  Du _{m} - (-A)^{1/2}  Du _{}   \rangle \big)(x)dt.
\end{array}
\end{gather*}   
 Using the uniform bound on  $\| (-A)^{1/2}  Du _{m} \|_0$, the fact that 
    $([\pi_m F   \circ \pi_m - z_m] - [    F - z]     )$ and $(f_m - f)$ are uniformly bounded and both converge pointwise to 0 as $m \to \infty$  we obtain
\begin{gather*}
 \int_0^{\infty} e^{-\lambda t }   P_{t}^{(z)} \big(  
 f_m - f
 + \langle     (\pi_m F   \circ     \pi_m  - z_m) -    (  F -z)   
 , (-A)^{1/2}  Du _{m}   \rangle \big) (x) dt \to 0,
\end{gather*} 
 $x \in H,$
$\text{as}\, m \to \infty.$  It remains to prove that  
 \begin{gather*}
\lim_{m \to \infty} \,  \int_0^{\infty} e^{-\lambda t }   P_{t}^{ (z )}  \big(
  \langle      F    - z  , (-A)^{1/2}  Du _{m} - (-A)^{1/2}  Du _{}   \rangle \big) (x)dt  =0,\;\; x\in H.
\end{gather*}
Using the bounds  \eqref{e66} and Theorem \ref{ss13} (see also  \eqref{w12})    we can apply the dominated convergence theorem and obtain the assertion.
 \end{proof}

\subsection{ Weak uniqueness  
 when $\| F- z \|_0 < 1/4 $   }

Here we will apply the  regularity results of the previous section to obtain  
 \begin{lemma}\label{loc}   Let $x \in H$ and consider the  SPDE \eqref{sde}. 
   If there exists   $z \in H$  such that \eqref{111} holds 
 then we have uniqueness in law for \eqref{sde}. 
\end{lemma} 
\begin{proof} By Section 4,  for any $x \in H$, there exists a weak mild solution starting at $x \in H$. Equivalently, by Proposition \ref{ser}, for any $x \in H$, 
 there exists a  solution to the martingale problem for $({\cal L}, \delta_x)$.

We will prove   that given  two weak  mild solutions $X$ and $Y$ which both solve \eqref{sde} and start at $x$ we have that the law of $X_t$ coincides with the law of $Y_t$ on ${\mathcal B}(H)$, for any $t \ge 0$. By Theorem \ref{ria} we will deduce that $X$ a $Y$ have the same law on ${\mathcal B}(C([0, \infty);H)$.
  
Let us fix $x \in H$ and let $X = (X_t)  $ be a weak mild solution starting at $x \in H$.  We proceed in two steps.   First we prove useful formulas for finite-dimensional approximations of $X_t$ and then we pass to the limit obtaining a basic  identity for $X.$
\\ 
{\it Step 1. Some useful formulas for finite-dimensional approximations of $X_t$. }

For any $m\in\mathbb N$ we set $X_{t, m}:=\pi_m X_t,$
where    $\pi_m=\sum_{j=1}^me_j\otimes e_j$ (cf. formula \eqref{pp1}). We have
 \begin{equation*} 
 X_{t, m}=e^{tA} \pi_m  x +\int_{0}^{t}e^{\left(  t-s\right)  A} (-A_m)^{1/2}F^{}(X(s))
ds+\int_{0}^{t} \pi_m e^{\left(  t-s\right)  A}dW_{s},\;\;\; t \ge 0,
\end{equation*}
where  $A_m=A\pi_m $. Writing $\pi_m W_t = \sum_{k=1}^m W_t^{(k)} e_k$ it follows that 
\begin{equation} \label{maga} 
X_{t, m}
= \pi_mx+\int_0^tA_mX_{s}ds+\int_0^t  (-A_m)^{1/2}F(X_{s})ds+\pi_m W_t.
 \end{equation}
Let $f \in C^2_b (H)$.  As in \eqref{e3351} and \eqref{344}  we denote by $u _{nm }$ the classical solution of the equation
\begin{equation}
\label{e334s}
\lambda u _{nm}- L u _{nm}-\langle  (-A)^{1/2} \pi_m F_n \circ \pi_m, Du _{nm}   \rangle= f  \circ \pi_m,  \;\; \lambda \ge 1.
\end{equation} 
   Applying a finite-dimensional It\^o's formula to $u_{nm}(X_{t, m})
= u^{}_{nm}(X_{t})$ yields
\begin{equation}
\label{e35}
\begin{array}{lll}
du_{nm}(X_{t, m})=\frac12\;\mbox{\rm Tr}\;[D^2u _{n m}(X_{t, m})]dt
\\ \\ 
+\langle Du _{n m}(X_{t, m}), A_m X_t+ (-A_m)^{1/2}  F(X_t) \rangle dt
+\langle Du_{nm}(X_{t, m}), \pi_md W_t  \rangle.
\end{array}
\end{equation}
On the other hand, by \eqref{e334s} we have   
$$
\begin{array}{l}
\ds \lambda u _{nm}(X_{t, m})-\frac12\;\mbox{\rm Tr}\;[D^2u _{nm}(X_{t, m})]
\\ \ds
-\langle Du _{nm}(X_{t, m}), A_m X_{t, m}+ (-A_m)^{1/2}   F_n(X_{t, m}) \rangle= f (X_{t, m}). 
\end{array}
$$
 Taking into account \eqref{e35} and the fact that $u_{nm}  (\pi_m y) = u_{nm}  (y), $ $
y \in H,$ $\; n,m \ge 1$,
yields
   \begin{gather*}
u _{nm}(X_{t})  - u _{nm}(x) 
=\lambda \int_0^t  u _{nm}(X_{s})ds -
\int_0^t
f (X_{s, m})ds \\
+ \int_0^t \langle  (-A)^{1/2}  Du _{nm}(X_{s}),  (F(X_s)-F_n(X_{s, m}) )
\rangle ds + \int_0^t \langle Du _{nm}(X_{s}), \pi_m dW_s \rangle,  
\end{gather*}            
 $t \ge 0.$  
  By \eqref{2ee}   we deduce that   $(\int_0^t \langle Du _{nm}(X_{s}), \pi_m  dW_s \rangle)_{t \ge 0}$ is a martingale. Hence  
  \begin{gather} \label{sino}     
\E [u _{nm}(X_{t})]  - u _{nm}(x)
= \lambda \int_0^t  \E [u _{nm}(X_{s})]ds -
\int_0^t
\E [f (X_{s, m})]ds 
\\ \nonumber 
+ \int_0^t \E [ \langle  (-A)^{1/2}  Du_{nm} (X_{s}),  (F(X_s)-F_n(X_{s, m}) )
\rangle] ds.    
\end{gather} 
{\it Step 2. Passing to the limit in \eqref{sino} as $n, m \to \infty$.}

We apply the convergence   results  of Lemmas \ref{stima2} and \ref{stima3}.   To this purpose note the pointwise convergence 
$$
 \begin{array}{l}
 \pi_m  F_n \circ \pi_m   \to F
\end{array}
$$
first as $n \to \infty$ and then as $m \to \infty$ (according to the convergence used in the  previous section). Moreover $\sup_{n,m \ge 1} \|\pi_m  F_n \circ \pi_m  \|_0$ $\le \| F\|_0$ and  $u_{m}  (\pi_m y) = u_{m}  (y),$ $
y \in H,\;$ $  m \ge 1$. 
  Let us fix $m \ge 1$.  First we can pass to the limit as $n \to \infty$ in   \eqref{sino} by the Lebesgue convergence theorem and   get  
\begin{gather*}
 \begin{array}{l} \ds
\E [u _{m}(X_{t})]  - u _{m}(x)
  = \lambda \int_0^t  \E [u _{m}(X_{s})]ds -
\int_0^t
\E f (X_{s,m})ds 
\\ \ds
+ \int_0^t \E [\langle (-A)^{1/2} Du _{m}(X_{s}),   (F(X_s)-F\circ \pi_m (X_{s}) )
\rangle ]ds.  
 \end{array} 
\end{gather*}
Then, using also Lemma \ref{stima3}, we pass to   the limit as $m \to \infty$ and
   arrive at     
 $$
  \begin{array}{l}
  \E [u _{}(X_{t})]  - u _{}(x)
=  \lambda \int_0^t \E [u _{}(X_{s})]ds -
\int_0^t
\E [f (X_{s})]ds.  
\end{array} 
$$
Integrating both sides over $[0, \infty)$ with respect to $e^{\lambda t} dt$ and using the Fubini theorem we arrive at the basic identity
$$
 \begin{array}{l}
    u _{}(x)
=     \int_{0}^{\infty}  e^{- \lambda s}  
\E [f (X_{s})]ds,  \; \lambda \ge 1.  
\end{array} 
$$ 
 Now if  $Y$ is another weak mild  solution 
    starting at $x$ and defined on $(\tilde
\Omega,$ $ \tilde {\mathcal F},
 (\tilde {\mathcal F}_{t}), \tilde \P) $.
  We obtain, for any $f \in C^2_b(H)$, 
$$
  \int_{0}^{\infty}  e^{- \lambda s}  
 \E [f (X_{s})]ds =   \int_{0}^{\infty}  e^{- \lambda s}  
\tilde \E [f (Y_{s})]ds, \;\; \lambda \ge 1.
$$
By the  uniqueness of the Laplace transform and using an approximation argument we find that 
$ \E [g (X_{s})] = \tilde \E [g (Y_{s})]$, for any $g \in C_b (H)$, $s \ge 0$.
 Applying Proposition \ref{ser} and Theorem \ref{ria} we find that $X$ and $Y$ have the same law on  ${\cal B}(C([0, \infty); H))$. 
\end{proof}

\subsection {Weak uniqueness when  $F \in C_b(H,H)$ }

 Here we prove uniqueness using the localization principle (cf. Theorem \ref{uni1}  and  Lemma \ref{loc}).

\begin{lemma}\label{loc1}   Let $x \in H$ and consider the  SPDE \eqref{sde}. 
   If $F \in C_b(H,H)$
  then we have uniqueness in law for \eqref{sde}.
\end{lemma} 
\begin{proof} By Proposition \ref{ser} it is enough to show that 
 the martingale problem for $\L$  is well-posed (cf. \eqref{ll}).  
 By  Section 4,  for any $x \in H$, there exists a solution to the martingale problem for $(\L, \delta_x)$.

 In order to apply Theorem \ref{uni1}  we proceed into  two  steps.
In the first step we construct a suitable covering of $H$;  in the second step we define  suitable operators ${\L}_j$ according to  Lemma \ref{loc} such that 
 the martingale problem associated to each ${\L}_j$ is well-posed.     

 \smallskip
\noindent
{\it I Step.} There exists
     a  countable set of points $(x_j) \subset H$,
$j \ge 1$, and numbers $r_j >0$   with the following properties:

\smallskip
(i) the open balls $U_j= B(x_j, \frac{ r_j}{2})= \{ x \in H \, :\, |x- x_j|_H < r_j/2 \}$ 
  form a covering for $H$;

(ii)   we have:
$
\| F(x) - F(x_j) \| < 1/4, \;\;\;  x \in B(x_j, r_j).
$
\\
Using   the continuity of $F$: for any $x$ we  find $r(x)>0$ such that 
$$
 \begin{array}{l}
|F(y) - F(x)|_H < 1/4, \;\;\;\;  y \in B(x, r(x)).
 \end{array}
 $$   
 We have a covering  $\{ U_{x}\}_{x \in H}$ with $U_x = B(x, \frac{r(x)}{2})$.
Since $H$ is a separable Hilbert space we can choose a countable subcovering $(U_j)_{j \ge 1}, $ with $U_j = B(x_j, \frac{r(x_j)}{2})$ $= B(x_j, \frac{r_j}{2})$.

\smallskip
\noindent
\textit{II Step.} We construct $\L_j$ in order to apply the localization principle.
  
   Let us consider the previous covering $(B(x_j, r_j/2))$.   We take  $\rho \in C_0^{\infty}(\R_+)$, 
  $0 \le \rho \le 1$, $\rho(s)=1$, $0 \le s  \le 1$, $\rho(s) =0$ for $s \ge 2$.
   Define 
   $$
   \rho_j (x) = \rho \big ( 4\,  r_j^{-2}\, |x- x_j|^2_H \big), \;\; x \in H.     
 $$ 
Now  $\rho_{j} =1$ in $B(x_j, \frac{r_j}{2})$ and  $\rho_j =0$ outside $B(x_j,   
 \, r_j )$. Set
 $
 F_j(x) :=  \rho_{j}(x) F(x)  +  (1 - \rho_{j}(x))F(x_{j} ),$ $ x\in H,$ so that
   \begin{equation*}
 \sup_{x \in H}| F_j(x) - F(x_j) |_H  =  \sup_{x \in B(x_j, r_j)} \, | F(x) - F(x_j) |_H < 1/4 
 \end{equation*}
and  $F_j(x) = F(x)$, $x \in B(x_j, \frac{r_j}{2}) = U_j$. Define  
$D(\L_j) = C^2_{cil}(H)$, $j \ge 1$, and   
$$
\L f_j (x) = \frac{1}{2} Tr(D^2 f(x)) + \langle x, ADf(x) \rangle +
\langle (-A)^{1/2} F_j(x), Df(x) \rangle,\;\; f \in C^2_{cil}(H),\; x \in H.
$$
  We have 
 $\L_j f(x) =  {\L} f(x),\;\; x \in U_j,\;\; f \in C^2_{cil}(H)$ and 
  the martingale problem  for each $\L_j$,
  is well-posed by Lemma  \ref{loc} (with $F= F_j$ an\ $z = F(x_j)$).  By Theorem \ref{uni1} we find the assertion.   
   \end{proof}

\section{Proof of weak  uniqueness of Theorem \ref{base}
}   


Here 
we  
prove uniqueness in law for \eqref{sde} assuming that $F: H \to H$ 
 is continuous and has at most linear growth, i.e., it verifies \eqref{lin1}.
 To this purpose  we will use Lemma \ref{loc1} and Theorem  \ref{key}.  

\vskip 1 mm 
Let $X= (X_t)_{t \ge 0}$  be a mild solution of \eqref{sde} starting at $x \in H$ (under the assumption \eqref{lin1}) defined on some filtered probability space $(
\Omega,$ $ {\cal F},
 ({\cal F}_{t}), \P) $ on which it is defined a
 cylindrical ${\cal F}_{t}$-Wiener process $W$; see Section 4. For a cylindrical function $f \in C^2_{cil}(H) $ in general $\L f$ (see \eqref{ll}) is not a bounded function on $H$ because  $F$ can be unbounded. However
 we  know by a finite-dimensional It\^o's formula that 
 \begin{equation} \label{ffu}
  \begin{array}{l}
   M_t(f) =  f(X_t) - \int_0^t \L f(X_s)ds = f(x) + \int_0^t Df(X_s)dW_s 
 \end{array} 
 \end{equation}
 is still   a continuous square integrable ${\cal F}_t$-martingale. Note that we can apply It\^o's formula  because there exists $m \ge 1$ such that  $f(x) = f(\pi_m x)$, $x \in H$, and so  $f(X_t) = f(\pi_m X_t)$ (cf. formula  \eqref{maga}). 
 
\vskip 0.5 mm   
 Now let us consider  $B(0,n) = \{x \in H \, :\, |x|_H <n \}$ and define  {\sl continuous and bounded functions} $F_n : H \to H$ such that
$
F_n (y) = F(y),$ $ y \in B(0,n), 
$ $n \ge 1.  $   

\vskip 0.5 mm 
 To this purpose one can take $\eta \in C_0^{\infty}(\R)$ such that $0 \le \eta(s) \le 1$, $s \in \R$,  $\eta(s)=1$ for $|s| \le 1$ and  $\eta(s) =0 $ for $|s|\ge 2,$ and set
$
F_n(y) = F(y)\,  \eta \big(\frac{ |y|_H}{n}\big), $ $ y \in H.
$
Define  
$$ 
 \begin{array}{l}
\L_n f (y) = \frac{1}{2} Tr(D^2 f(y)) + \langle y, A Df(y) \rangle +
\langle F_n (y),  (-A)^{1/2} Df(y) \rangle, \;\; f \in C^2_{cil}(H), \, y \in H.
\end{array}  
$$   
Let us introduce the exit time   $\tau_n^X = \inf \{ t \ge 0 \, : \, |X_t|_H \ge n \}$ ($\tau_n^X = + \infty$ if the set is empty; cf. \eqref{ta1}) for each $n \ge 1$.
 It is an ${\cal F}_t$-stopping time (cf. Proposition II.1.5 in \cite{EK}).
By the optional sampling theorem (cf. Theorem II.2.13 in \cite{EK})  we know that 
\begin{gather*}
M_{t \wedge \tau_n^X}(f)=  f( X_{t \wedge \tau_n^X}) - \int_0^{t \wedge \tau_n^X} \L f(X_s)ds
= f(X_{t \wedge \tau_n^X}) - \int_0^{t \wedge \tau_n^X} \L_n f(X_{s   \wedge \tau_n^X})ds, \;\; t\ge 0,
\end{gather*}
 is a  martingale with respect to the filtration $({\mathcal F}_{t \wedge \tau_n^X})_{t \ge 0}$; note that the process $(X_{t \wedge \tau_n^X})_{t \ge 0}$ is adapted with respect to $\big ( {\mathcal F}_{t \wedge \tau_n^X} \big )$ (see Proposition II.1.4 in \cite{EK}).

Thus  $(X_{ t\wedge \tau_n^X  })_{t \ge 0} $ is a solution to  the  {\sl stopped martingale problem for $(\L_n, \delta_x, B(0,n)$).}
 By Lemma \ref{loc1}   {\sl the martingale problem for each $\L_n$ is well-posed because $F_n \in C_b(H,H)$}.   By Theorem \ref{key}  also the   stopped martingale problem for   $(\L_n, \delta_x, B(0,n)$) is well-posed, $n \ge 1$. 

\vskip 0.5 mm
Let  $Y$ be another  mild solution starting at $x \in H$. Then $(Y_{ t\wedge \tau_n^Y  })_{t \ge 0} $ also solves the stopped martingale problem for $(\L_n, \delta_x, B(0,n))$. 
  By weak uniqueness of the stopped martingale problem it follows that, for any $n \ge 1$,  $(X_{ t\wedge \tau_n^X  })_{t \ge 0} $ and $(Y_{ t\wedge \tau_n^Y  })_{t \ge 0} $ have the same law. Now it is not difficult to prove that  $X$ and $Y$ have the same law on ${\cal B}(C([0,\infty);  H))$ and this finishes the proof.   

 { 
\section{An extension to locally bounded functions $F: H\to H$}  
 
 Assuming weak existence for \eqref{sde} one can obtain the following extension of  Theorem
  \ref{base}.

  \begin{theorem}
\label{extension} 
 Let us consider \eqref{sde} under Hypothesis \ref{d1} and  fix  $x \in H$. Assume   that 
 
 
  H1) $F: H \to H$ is continuous and  bounded on bounded sets of $H$;  
 
 \vskip 1mm
 
 H2)  there exists a weak mild solution $(X_t)_{t \ge 0 }$ of \eqref{sde}
 starting at $x \in H$.   
 
\vskip 0,5 mm 
 Under  the previous assumptions weak uniqueness  holds, i.e.,  all  weak mild solutions starting at  $x \in H$ have the same law on ${\cal B}(C([0, \infty); H)$.      
\end{theorem}
  \begin{proof} The proof is similar to the one of Section 6.
   We give
    some details  
   for the sake of completeness.  
   Let $X= (X_t)_{t \ge 0}$  be a mild solution of \eqref{sde} starting at $x \in H$ (defined on some filtered probability space $(
\Omega,$ $ {\cal F},
 ({\cal F}_{t}), \P) $).
 We   have that $M_t(f)$ in \eqref{ffu} is   a continuous square integrable ${\cal F}_t$-martingale, for any 
   $f \in C^2_{cil}(H)$.
    Using  the continuity and the local boundedness of $F$, we obtain that  the functions $
F_n(y) = F(y)\,  \eta \big(\frac{ |y|_H}{n}\big), $ $ y \in H,
$ are continuous and bounded from $H$ into $H.$
 
 By the optional sampling theorem 
  we find that  $(X_{ t\wedge \tau_n^X  })_{t \ge 0} $ is a solution to  the  { stopped martingale problem for $(\L_n, \delta_x, B(0,n)$).}
  Using Lemma \ref{loc1} and Theorem  \ref{key}     we know that  the   stopped martingale problem for   $(\L_n, \delta_x, B(0,n)$) is well-posed, $n \ge 1$.  Proceeding as in the final part  of Section 6 we obtain the assertion.
\end{proof}

\subsection{Singular perturbations of classical stochastic Burgers equations}  

Here we show that  Theorem \ref{extension} can be applied to  SPDEs \eqref{sde} in cases when  $F$ grows more than linearly. As an example we consider   
  \begin{gather} \label{bur}
d u (t, \xi)=   \frac{\partial^2}{\partial   \xi^2}  u(t, \xi)dt +   \, { h( }u(t, \xi) )\cdot  g \big (\,  \big |  u(t, \cdot ) \big |_{H^1_0 } \big) dt     
\nonumber 
+ \frac{1}{2} \frac{\partial }{\partial \xi} \big ( u^2(t, \xi) \big ) dt +
\, \sum_{k \ge 1 }  
\frac{1}{k}   dW_t^{(k)}  {e_k(\xi)},  
\\  \;\; u(0, \xi) = u_0(\xi), \;\;\; \xi \in (0,\pi),
\end{gather}
 $u(t,0) = u(t,\pi)=0$, $t >0$, $u_0 \in H^1_0(0,\pi)$;  {\sl $g: \R \to \R $ is continuous and $h : \R \to \R$ is  a $C^1$-function. Moreover to get existence of solutions we require  } 
 \begin{equation}\label{er3}
 \sup_{s \in \R} |h'(s)| \, \cdot  \sup_{s \in \R} |g(s) | \le 1. 
\end{equation}
For instance, we can consider $  h( u(t, \xi) )\cdot  g \big (\,  \big |  u(t, \cdot ) \big |_{H^1_0 } \big) = u(t, \xi) \cdot  $ $\big(\sqrt{ |  u(t, \cdot ) 
|_{H^1_0 } } \,\,\,  \wedge 1 \big)$.

  Recall that $
 e_k (\xi) = \sqrt{2/\pi} \,  \sin (k \xi), $ $ \xi \in [0, \pi],\;\; k \ge 1
 $ (cf. Section 2; note that in \eqref{bur} the  noise  is ``more regular'' than  the one in \eqref{bur0}).

\vskip 1mm

 We first  establish existence of  mild solutions with values in $H^1_0(0,\pi)$ when $g=0$ (see Proposition \ref{dap}).  This is needed in other to show that the classical Burgers equation can be considered in the form \eqref{sde}
  with  a suitable $F = F_0 : H^1_0(0,\pi) \to H^1_0(0,\pi) $ continuous and locally bounded (see  \eqref{mq144}).
  To this purpose we follow the approach in  Chapter 14 of \cite{ergodicity}.
 
 Then to get  well-posedness of \eqref{bur} (see Proposition \ref{well1}) we will apply the Girsanov theorem using   an exponential estimate proved in \cite{burgers}. Such  Girsanov theorem  provides existence of weak solutions (cf. Remark \ref{che}). Uniqueness in law is  obtained direclty using Theorem  \ref{extension}.

 \vskip 1mm 
 We need to review    basic facts about fractional powers of the  
operator  
$A = \frac{d^2}{d  \xi^2}$ with Dirichlet boundary conditions, i.e., $D(A) = H^2(0, \pi) \cap H^1_0 (0, \pi)$ (cf. Section 2).  The
eigenfunctions are
 $
 e_k (\xi),$ $  k \ge 1,
 $ with 
  eigenvalues $- k^2$ (we set
  $\lambda_k =  - k^2 $). 
 For $v \in L^2(0, \pi)$ we write $v_k = \langle v,  e_k\rangle$ $= \int_0^{\pi} v(x) e_k (x) dx $, $k \ge 1$.

We introduce  for $ s >0 $  the Hilbert spaces 
  \begin{equation}\label{s335}
   \begin{array}{l}
{\mathcal H}_{s} = D((-A)^s) = \big \{ u \in L^2(0, \pi) \, : \, \sum_{k \ge 1} \lambda_k^{2s} \,  u_k^2 =\sum_{k \ge 1} k^{4s} \,  u_k^2  < \infty \big \}.
\end{array} 
\end{equation}
Moreover, for any $u \in {\mathcal H}_s$, $(-A)^s u = $ $
 \sum_{k \ge 1} k^{2s} \,  u_k e_k  $. 
 He also set ${\mathcal H}_0 = L^2(0, \pi)$.
  We have $\langle  u,v \rangle_{{\mathcal H}_s} = \sum_{k \ge 1} k^{4s} \,  u_k v_k$  (note that $|u|_{L^2} \le |(-A)^s u|_{L^2}= |u|_{{\mathcal H}_s}$, $u \in {\mathcal H}_s $, $s >0$).  

 \vskip 1mm  If $u \in H^1_0 (0,\pi)   $,  $|u|_{H^1_0(0,\pi) } = |u'|_{L^2(0,\pi) }$, where $u' $ is  the weak derivative of $u$.   We have   
\begin{gather}\label{h1o}  
{\mathcal H}_{1/2} =  H^1_0(0,\pi) \;\; \text{with equivalence of norms;  }
\\ 
%
\label{sob1}   
{\mathcal H}_{1/8} \subset L^4(0,\pi)
\end{gather}
(with continuous inclusion, i.e., there exists $C>0$ such that $|u|_{L^4} \le C |u|_{{\mathcal H}_{1/8}}$, $u \in {\mathcal H}_{1/8}$). 
 Assertion  \eqref{sob1} follows by a classical 
 Sobolev embedding theorem (cf. Theorem 6.16 and Remark 6.17 in \cite{hairer}). We only note that if $u \in {\mathcal H}_{1/8}$ one  can consider the odd extension $\tilde u$ of $u$ to $(- \pi , \pi)$; it is easy to check that $\tilde u$ belongs to the space $ H^{1/4}(-\pi, \pi) $ considered in \cite{hairer}. 
 
 \vskip 1mm 
 We also have with  continuous inclusion (cf. Lemma 6.13 in \cite{hairer})
 \begin{equation}\label{sob2}
{\mathcal H}_{s} \subset \{ u \in C([0, \pi]),\; u(0)=u(\pi)=0\}, \;\;\; s > 1/4.
\end{equation}
Now let us consider the linear bounded operator $T : {\mathcal H}_{1/2} \to {\mathcal H}_{1/2}$, $T u = (-A)^{-1/2} \partial_{\xi} u $, $u \in {\mathcal H}_{1/2}$;  $T$ can be extended to a linear and bounded operator   $T : {\mathcal H}_0 = L^2(0, \pi)\to {\mathcal H}_0$, see  Section 2.0.1. 
By interpolation  it follows that 
\begin{equation}\label{inter}
 T =(-A)^{-1/2} \partial_{\xi} \;\; \text{is bounded linear operator from    $\; {\mathcal H}_{s}$ into $ {\mathcal H}_s$, $s \in [0,1/2]$}.
\end{equation}
Indeed by Theorem 4.36 in \cite{interpola} we know that 
  ${\mathcal H}_{s/2} $ can be identified with the real interpolation space $({\mathcal H}_0, {\mathcal H}_{1/2})_{s,2}$, $s \in (0, 1)$. Applying Theorem 1.6 in \cite{interpola}   we deduce \eqref{inter}.   

\vskip 1mm Let $T>0$.  For $g \in C([0,T]; {\mathcal H}_0)$ we define $(Sg)(t) =  \int_{0}^{t} e^{(  t-s)  A} g  (s) ds$, $t \in [0,T]$. 
One can prove that 
 $Sg \in C([0,T]; {\mathcal H}_s)$, for any  $s \in [0,1)$. More precisely, 
\begin{equation}\label{bou1}
S \; \text{is a bounded linear operator from $C([0,T]; {\mathcal H}_0)$ into $C([0,T]; {\mathcal H}_s)$, $s \in [0,1)$.}
\end{equation}
 This result can be also deduced from Proposition 5.9 in \cite{DZ} with $\alpha =1$, $E_1 = {\mathcal H}_s$ and $E_2 = {\mathcal H}_0$. We only remark that, for any $p>1$, $L^p (0,T; {\mathcal H}_0) \subset C([0,T]; {\mathcal H}_0)$ (with continuous inclusion) and  $|(-A)^{s} e^{tA} x|_{{\mathcal H}_0}$ $=| e^{tA} x|_{{\mathcal H}_s} \le \frac{C}{t^{s}} |x|_{{\mathcal H}_0}$ (see Proposition 4.37 in \cite{hairer}).

\vskip 1mm      
In the next proposition, assertion (i) extends  a result   of \cite{ergodicity} which actually shows the existence of a mild solution to the stochastic Burgers equations with continuous path in ${\mathcal H}_s$, $s \in (0,1/4)$.        
 Assertion (ii) is proved in \cite{burgers}.
 \begin{proposition} \label{dap}
 Let us consider \eqref{bur} with $g=0$.  Then the following assertions hold:
 
 i) for any $u_0 \in {\mathcal H}_{1/2}$ there exists a pathwise unique mild solution $Y = (Y_t) = (Y_t)_{t \ge 0 }$ with continuous paths in ${\mathcal H}_{1/2}$. 
 
 ii)  The following estimate holds, for any $T>0,$
 \begin{equation}\label{esp}
 \E \Big [ \exp \Big ( {\frac{1}{2} \int_0^T \big |Y_s \big |_{{\mathcal H}_{1/2}}^2 ds} \Big )\Big]  < \infty .
\end{equation}
 \end{proposition}
\begin{proof} 
 According to \cite{ergodicity} and  \cite{burgers}, setting   $u(t, \cdot ) = Y_t$ we write \eqref{bur}  with $g=0$ as 
\begin{equation}\label{mil3}
 \begin{array}{l}
  Y_t = e^{tA} u_0 \, + \, \frac{1}{2}\int_{0}^{t} e^{(  t-s)  A} \, \partial_{\xi}  (
Y_{s}^2)
ds
+\int_{0}^{t}e^{(  t-s)  A} \sqrt{C} dW_{s},\;\; t \ge 0,
\end{array} 
\end{equation}
where $W_t = \sum_{k \ge 1 }  
   W_t^{(k)} e_k $ is a cylindrical Wiener process on ${\mathcal H}_0 = L^2(0,\pi)$ and 
   $C= (-A)^{-1}: $ ${\mathcal H}_0 \to {\mathcal H}_0 $ is  symmetric, non-negative
and of trace class,
 $C e_k = \frac{1}{k^2} e_k$, $k \ge 1$.  
  
 \vskip 1mm 
\noindent {\bf (i)} In Theorem 14.2.4 of \cite{ergodicity} (see also the references therein) it is proved that, for any $T>0,$ there exists a pathwise unique solution $Y$ to \eqref{mil3} on $[0,T]$ such that, $\P$-a.s., $Y \in C([0,T]; {\mathcal H}_0) \cap 
L^2(0,T; {\mathcal H}_{1/2})$ (i.e., $\P$-a.s, the paths of $Y$ are continuous  with values in ${\mathcal H}_0$ and square-integrable with values in ${\mathcal H}_{1/2}$); 
such  result  holds even  if we replace $C $ by identity $I$. 
 By a standard argument based on the pathwise uniqueness, we get a solution $Y$ defined on $[0, \infty)$ which verifies  $Y \in C([0,\infty); {\mathcal H}_0) \cap 
L^2_{loc}(0, \infty; {\mathcal H}_{1/2})$, $\P$-a.s.. 

Let us fix $T>0$. To prove our assertion, we will show that
\begin{equation}\label{s119}
 Y \in C([0,T]; {\mathcal H}_{1/2}), \;\;\; \P\text{-a.s.}.
\end{equation}
Note that the stochastic convolution $W_A(t)= \int_{0}^{t}e^{\left(  t-s\right)  A} \sqrt{C} dW_{s}$ has a modification with continuous paths with values in ${\mathcal H}_{1/2}$
(to this purpose one can use Theorem 5.11 in \cite{DZ}).  
  Moreover in  Lemma 14.2.1 of \cite{DZ1} it is proved that the operator $R$,
\begin{equation}\label{dz1}
\begin{array}{l}
(Rv)(t) = \int_{0}^{t} e^{(  t-s)  A} \, \partial_{\xi}  
v(s)ds ,\;\;\; t \in [0,T],\;   v \in C([0,T]; {\mathcal H}_{1/2}),
\end{array} 
\end{equation}
can be extended to a linear and bounded operator from $C([0,T]; L^1(0, \pi))$ 
 into $C([0,T]; {\mathcal H}_s)$, $s \in (0, 1/4)$. Since the mapping: $ h \mapsto h^2$ is continuous from $C([0,T]; {\mathcal H}_0)$ into $C([0,T]; L^1(0, \pi))$, we obtain that 
\begin{gather*}
 \begin{array}{l}
u \mapsto R (u^2) \;\; \text{is continuous from $C([0,T]; {\mathcal H}_0)$ into $C([0,T]; {\mathcal H}_s)$.}
 \end{array}
\end{gather*} 
We deduce from  Lemma 14.2.1 of \cite{DZ1} that the solution $Y \in C([0,T]; {\mathcal H}_{s}),   \P$-a.s., $s \in (0,1/4)$. To get more spatial regularity for $Y$  we proceed in two steps.

\vskip 1mm 
\noindent {\it I Step.}  We show that, $ \P$-a.s,  $Y \in C([0,T]; {\mathcal H}_{s})$, $s \in (0,1/2)$.

 Let us fix $s = 1/8$.  By \eqref{sob1} we know that the mapping: $ h \mapsto h^2$ is continuous from $C([0,T]; {\mathcal H}_{1/8})$ into $C([0,T]; {\mathcal H}_0)$. Moreover, using \eqref{inter} we can write, for $w \in C([0,T]; {\mathcal H}_0), $
\begin{gather*}
(Rw)(t) = \int_{0}^{t} e^{(  t-s)  A} \, \partial_{\xi}  
w(s) = \int_{0}^{t} e^{(  t-s)  A} \, (-A)^{1/2}  \, [(-A)^{-1/2} \partial_{\xi}]  
w(s) 
ds, \; t \in [0,T]. 
\end{gather*}
Note that    $[(-A)^{-1/2} \partial_{\xi}]  
w  \in C([0,T]; {\mathcal H}_0)$. 
 By \eqref{bou1} we know that, for any $\epsilon \in (0,1)$,  
 $$ t \mapsto (-A)^{1- \epsilon}\int_{0}^{t} e^{(  t-s)  A}  \, [(-A)^{-1/2} \partial_{\xi}]  
w(s)  ds$$ belongs to $C([0,T]; {\mathcal H}_0)$. Hence 
$$
(-A)^{s} Rw \in C([0,T]; {\mathcal H}_0),\;\; s \in (0, 1/2),\; i.e.,\;  
Rw \in C([0,T]; {\mathcal H}_s), \;\; s \in (0, 1/2).
$$ 
  Using this fact we easily  obtain  that,
  $ \P$-a.s,  $Y \in C([0,T]; {\mathcal H}_{s})$, $s \in (0,1/2)$.
 
\vskip 1mm 

\noindent  {\it II Step.}  We show that $Y \in C([0,T]; {\mathcal H}_{1/2}),  $ $\P$-a.s..  
 
\vskip 1mm 
 Let us fix $s \in (1/4, 1/2)$ and  recall \eqref{sob2}. According to \cite{grisvard} the space  ${\mathcal H}_s$ can be identified with $\{ u \in W^{2s,2}(0, \pi) \, :\,  
 u(0)=u(\pi)=0\}$, where 
\begin{equation*}\label{gris}
 \begin{array}{l}
W^{2s,2}(0, \pi) = \big \{ u \in {\mathcal H}_0 \, :\, [u]_{W^{2s,2}(0, \pi)}^2 
= \int_{0}^{\pi}  \int_{0}^{\pi}  |u(x) - u(y)|^2 \, |x-y|^{-1 - 4s} \, dx dy < \infty  \big \}
\end{array} 
\end{equation*}
 is  
 a Sobolev-Slobodeckij  space; the norm $|u|_{W^{2s,2}(0, \pi)} = |u|_{{\mathcal H}_0} $ $+ [u]_{W^{2s,2}(0, \pi)}$  is equivalent to $|u|_{{\mathcal H}_s}$
 (see also Theorem 3.2.3 in \cite{analityc}, taking into account that 
 ${\mathcal H}_s $ can be identified with the real interpolation space $({\mathcal H}_0, D(A) )_{s,2} $     by Theorem 4.36 in \cite{interpola}). 
 
 Using the previous characterization and \eqref{sob2} it is easy to prove that if  $u \in {\mathcal H}_s$ and $v \in {\mathcal H}_s$ then the pointwise product $u v \in {\mathcal H}_s $. Indeed we have 
  $$
   |u(x) v (x) - u(y) v (y)| \le \| u\|_{0} \, |v(x) - v(y)|
   +  \| v\|_{0} \, |u(x) - u(y)|, \;\; x,y \in [0, \pi],
 $$
 and so  $[u v]_{W^{2s,2}(0, \pi)} \le c |u|_{W^{2s,2}(0, \pi)}
  \, |v|_{W^{2s,2}(0, \pi)}
  \le c' |u|_{{\mathcal H}_s}\, |v|_{{\mathcal H}_s}$. It follows that 
   $|uv|_{{\mathcal H}_s} \le C |u|_{{\mathcal H}_s}\, |v|_{{\mathcal H}_s}$. 
   
 Let now $u  \in C([0,T]; {\mathcal H}_{s})$.  Using that $|u^{2}(t) - u^2(r)|_{{\mathcal H}_s} \le |u(t) - u(r)|_{{\mathcal H}_s} 
  |u^{}(t) + u(r)|_{{\mathcal H}_s} $ $\le  $ $2 |u|_{C([0,T]; {\mathcal H}_{s})} |u(t) - u(r)|_{{\mathcal H}_s}$, $t, r \in [0,T]$, we see that 
    the mapping: 
\begin{equation}
   u \mapsto u^2 \;\; \text{  is continuous from $C([0,T]; {\mathcal H}_{s})$ into $C([0,T]; {\mathcal H}_{s})$}.
\end{equation}
 Hence, 
taking into account I Step, to get the assertion it is enough to prove that 
\begin{gather} \label{da3}
R \eta \in C([0,T]; {\mathcal H}_{1/2}) \;\; \text{if} \; \eta \in C([0,T]; {\mathcal H}_{s}) , \;\; s \in (1/4, 1/2).
\end{gather}
This would imply   $R (\eta^2)  \in C([0,T]; {\mathcal H}_{1/2}) $  $\text{if} \; \eta \in C([0,T]; {\mathcal H}_{s})$ and so  $Y \in C([0,T]; {\mathcal H}_{1/2}),  $ $\P$-a.s..  
   Let us fix $\eta \in C([0,T]; {\mathcal H}_{s}).$
 Using \eqref{inter} we can write 
\begin{gather*}
(R\eta)(t) = \int_{0}^{t} e^{(  t-s)  A} \, \partial_{\xi}  
\eta (s) = \int_{0}^{t} e^{(  t-s)  A} \, (-A)^{1/2}  \, [(-A)^{-1/2} \partial_{\xi}]  
\eta (s) 
ds, \; t \in [0,T], 
\end{gather*}
 where $ [(-A)^{-1/2} \partial_{\xi}]  
\eta   \in C([0,T]; {\mathcal H}_s)$. Hence $
\theta(r)=  (-A)^s [(-A)^{-1/2} \partial_{\xi}]  
\eta(r) \in C([0,T]; {\mathcal H}_0)$. Writing
$$
(R\eta)(t) = \int_{0}^{t} e^{(  t-r)  A} \, (-A)^{1/2 -s}  \, \theta (r) dr, \;\; t \in [0,T],
$$
 and using 
 \eqref{bou1}, we find that $(-A)^{1/2} R \eta \in C([0,T]; {\mathcal H}_0)$
and this shows \eqref{da3}.
 
\vskip 2 mm \noindent {\bf (ii)  } A similar  estimate is proved in  Propositions 2.2 and 2.3 in \cite{burgers}. However in \cite{burgers} equation \eqref{mil3} is considered in $L^2(0,1)$ (instead of $ L^2(0, \pi)$); the authors prove 
that $\E \big [ e^{\epsilon  \int_0^T  |Y_s  |_{H^1_0(0,1)}^2 \,  ds}\, \big] $ $ < \infty$ if $\epsilon \le \epsilon_0 = \pi^2 / 2 \|C \|$ (using the operator   norm  $\|C \|$ of $C$).

The condition $\epsilon \le \epsilon_0$  is used in the proof of Proposition 2.2 in order to get the inequality  $- |x|^2_{H^1_0} + 2 \epsilon |\sqrt{C} x|_{L^2}^2 \le 0$, $x \in H^1_0$. In our case   $\epsilon_0 =1/2$ since $\| C\|=1$.
 \end{proof}
In the remaining part   we consider  
$$
\H = H^1_0 (0, \pi) = {\mathcal H}_{1/2}
$$  
as the reference Hilbert space and  study the SPDE \eqref{bur} 
 in $\H$.  
 
 We will  consider  the following restriction of $A:$ 
 \begin{equation} \label{aa2}
  \A = \frac{d^2}{d  \xi^2} \;\; \text{ with   $D(\A) = \big \{ u \in H^3(0, \pi) \; : \; u, \frac{d^2 u}{d  \xi^2}  \in  H^1_0 (0, \pi) \big \}$;} \;\;\; \A : D(\A) \subset \H \to \H.  
\end{equation}
Eigenfunctions of $\A$ are $\tilde e_k(\xi)= \sqrt{2/\pi} \, \frac{1}{k}  \sin (k \xi) =  
 \frac{1}{k} e_k(\xi)
$ with  eigenvalues  $- k^2$, $k \ge 1$.   

It is clear that $\A$ verifies Hypothesis  \ref{d1} when  $H = \H$. Moreover
 $(\frac{e_k}{k}) = (\tilde e_k)$ forms  an orthonormal basis in $\H$. 
 The noise in \eqref{bur} will be indicated by $\W$; it  is a 
   cylindrical Wiener process on $\H$:
   \begin{equation}\label{noise}
   \begin{array}{l} 
\W_t(\xi) = \sum_{k \ge 1 }  \frac{1}{k}
   W_t^{(k)}  {e_k(\xi)} = \sum_{k \ge 1 }   
   W_t^{(k)}  {\tilde e_k(\xi)},\;\;\; t \ge 0,\; \xi \in [0,\pi]. 
\end{array} 
\end{equation} 
Let $D_0$ be the space of infinitely differentiable functions vanishing in a 
neighborhood of $0$ and $\pi$. Such functions are dense in $\H$. The operator
\begin{equation}\label{s33}
(-\A )^{1/2} \partial_{\xi} : D_0 \to \H \; \text{can be extended to a bounded linear operator from $\H$ into $\H$}. 
\end{equation}
To check this fact we consider $y \in D_0$ and $x \in \H$. Define $x_N = \sum_{k = 1 }^N   
   x_k  {\tilde e_k}$, with $x_k = \langle x, \tilde e_k \rangle_{\H}$, $N\ge 1$.
  Using that $(- \A )^{1/2}$ is self-adjoint on $\H$ and integrating by parts we find 
\begin{gather*} 
\begin{array}{l}
 \langle (-\A )^{-1/2}  \partial_{\xi} \, y , x_N \rangle_{\H} = \langle   \partial_{\xi} y , (- \A)^{-1/2} x_N \rangle_{\H} =  
  \langle \partial_{\xi}^2    y , \partial_{\xi} \, \sum_{k = 1 }^N   
   \frac{x_k}{k}  {\tilde e_k} \rangle_{L^2(0,\pi)}    
  \\ \\
=   - \langle \partial_{\xi}    y , \partial_{\xi}^2 \, \sum_{k = 1 }^N   
   \frac{x_k}{k}  {\tilde e_k} \rangle_{L^2(0,\pi)}
    =
     \langle \partial_{\xi}     y ,   \sum_{k = 1 }^N   
    {x_k} { \sin (k \, \cdot)}  \rangle_{L^2(0,\pi)}. 
    \end{array}
  \end{gather*}   
  Hence $|\langle (-\A )^{-1/2}  \partial_{\xi} \, y , x_N \rangle_{\H}\, |$
  $\le |y|_{\H} \, $ $ (\sum_{k = 1 }^N   
    {x_k}^2)^{1/2}$ $\le |y|_{\H} \,|x|_{\H} $ 
   and 
  we  get the assertion. 
   Let us introduce, for any $x \in \H$,  
   \begin{equation} \label{f02}
   \begin{array}{l}
   F_0 (x) =  \frac{1}{2}  {\A}^{-1/2} \partial_{\xi} [x^2].
   \end{array}
   \end{equation}
   Since the mapping $x \mapsto x^2$ is continuous and locally bounded from  $\H$ into $\H$ (recall that $|x^2|_{\H} = 2 |x \, \partial_{\xi} x|_{L^2(0,\pi)}$)
it is clear that 
\begin{equation} \label{f0}
  F_0: \H \to \H  \;\; \text{is continuous and locally bounded.}
\end{equation}
The mild solution $Y $ of Proposition \ref{dap} with paths in $C([0,\infty); \H)$ verifies, $\P$-a.s.,
\begin{equation}
  \label{mq144}
Y_{t}=e^{t \A }x +\int_{0}^{t}(- \A)^{1/2}e^{(  t-s)  \A}F_0  (
Y_{s})ds + \int_{0}^{t}e^{(  t-s)  \A}d \W_{s};\;\; t \ge 0, 
\end{equation}
 where $\A$ is defined in \eqref{aa2} and we have set $u_0=x \in \H$.
  
  \smallskip 
  {\sl We consider the following  SPDE which includes 
    \eqref{bur} as a special case:}
 \begin{equation}
  \label{cheef}
X_t = e^{t \A }x +\int_{0}^{t}(- \A)^{1/2}e^{(  t-s)  \A}F_0  (
X_{s})ds
 + \int_{0}^{t} e^{(  t-s)  \A} B  (
X_{s})ds 
+
\int_{0}^{t}e^{(  t-s)  \A}d \W_{s},
 \end{equation}
$t \ge 0 $. Here  
\begin{equation}\label{nonl}
B: \H \to \H \; \; \text{is continuous and } \;\;  |B(x)|_{\H} \le c_0 +  |x|_{\H}, \;\; x \in \H,
\end{equation}
 for some $c_0 \ge 0.$ In \eqref{bur} we have 
 $B(x) = h(x)   g(| x|_{\H})$, $x \in \H$, and  \eqref{nonl} holds with $c_0 = 0$
 (we only note that  $|B(x)|_{\H}^2 \le \| g\|_{0}^2 \, \int_0^{\pi} |h '(x(\xi)) \cdot\,  \frac{dx}{d \xi}|^2 d \xi $ $\le \| g\|_{0}^2 \, \| h'\|_{0}^2 \, |x|_{\H}^2$ $\le  |x|_{\H}^2$, $x \in \H$).
  
  Condition \eqref{er3} is used to guarantee the bound on $B$ in \eqref{nonl}. This bound is   used to check the  Novikov condition \eqref{novikov} and prove the existence part in the following result.
\begin{proposition} \label{well1} 
 Let us consider \eqref{cheef} on    
 $\H= H^1_0(0, \pi)$ with 
  $\A$ given in \eqref{aa2} and the cylindrical Wiener process $\W$ on $\H$ given in \eqref{noise} ($(W^{(k)})_{k \ge 1
  }$ are independent real Wiener processes). 
  Let  $F_0$ as  in \eqref{f02} and  suppose that   $B: \H \to \H$   verifies  \eqref{nonl}. 
   Then the following assertions hold.
  
 i) For any $x \in \H$,  there exists a weak mild solution $(X_t)_{t \ge 0 }$.
 
ii)  Weak uniqueness holds for \eqref{cheef} for any $x \in \H.$
 \end{proposition}
\begin{proof} {\bf i)} Let us fix $x \in \H$. We will use the Girsanov theorem  as in Appendix A.1 of \cite{DFPR}, using the  reference Hilbert space $\H$.

 Let $Y = (Y_t)$ be the unique  solution to the Burgers equation  \eqref{mq144} with values in $\H$ and such that  $Y_0 =x$. This is defined on  a  filtered probability space 
 $(
\Omega, {\mathcal F},$ $
 ({\mathcal F}_{t}), \P )$  
  on which it is defined the
cylindrical Wiener process $\W$ on $\H$.
 Set 
\begin{gather*}
b(s) = B(Y_s),\;\; s \ge 0,
\end{gather*}
and note that $|b(s)|_{\H} \le c_0  +   |Y_s|_{\H}$, $s \ge 0$,   by \eqref{nonl}. The process   $(b(s))$ is  progressively measurable  and verifies $\E \int_0^T
|b(s)|_{\H}^2    ds < \infty $, $T>0$ (see \eqref{esp}  and recall that $e^r \ge  1 + r$). Moreover, by \eqref{esp} and \eqref{nonl}  it follows  that, for any $T>0,$ 
\begin{gather} \label{novikov}
\E  \big [ e^{\frac{1}{2} \int_0^T  |b(s)  |_{\H }^2 ds}\big]  \le \, C_T \, 
 \E  \big [ e^{\frac{1}{2} \int_0^T  | Y_s |_{\H }^2 ds}\big] < \infty.
\end{gather} 
Let ${U_t = \sum_{k \ge 1} \int_0^t \langle b(s) , \tilde e_k\rangle_{\H}  \, dW^{(k)}_s}$, $t  \ge  0$, and   
    fix $T>0.$ By Proposition 17 in \cite{DFPR} we know that $\tilde W^{(k)}_t = W^{(k)}_t - \int_0^t \langle \tilde e_k, b(s)\rangle_{\H} ds$, $t \in [0,T]$, $k \ge 1,$ are independent real Wiener processes on $(
\Omega, {\mathcal F},$ $
 ({\mathcal F}_{t}), \tilde \P )$, where 
  the probability measure 
 $$
 \tilde \P = e^{ U_T \,  
 - \, \frac{1}{2} \int_0^T  |b(s)  |_{\H }^2 ds} \, \cdot  \P
 $$ is equivalent to $\P$ (the quadratic variation process $\langle U\rangle_t$ $= \int_0^t \big |b(s) \big |_{\H }^2 ds$, $t \in [0,T]$). 
 
 Hence  $\tilde \W_t = \sum_{k \ge 1} \tilde W^{(k)}_t \tilde e_k$, $t \in [0,T]$, is a cylindrical Wiener process on $\H$ defined on $(
\Omega, {\mathcal F},$ $  
 ({\mathcal F}_{t}), \tilde \P )$. 
  Arguing  
as in Proposition 21 of \cite{DFPR} we obtain that 
\begin{gather*}
 \begin{array}{l}
Y_t = e^{t \A }x +\int_{0}^{t}(- \A)^{1/2}e^{(  t-s)  \A}F_0  (
Y_{s})ds
 + \int_{0}^{t} e^{(  t-s)  \A} B  (
Y_{s})ds 
+
\int_{0}^{t}e^{(  t-s)  \A}d \tilde \W_{s}, \;\; t \in [0,T],
 \end{array}
\end{gather*}
 $\P$-a.s.. 
 Thus $Y$ a mild solution on $[0,T]$ to \eqref{cheef} defined on $(
\Omega, {\mathcal F},$ $
 ({\mathcal F}_{t}), \tilde \P )$. 
 
 Since $T>0$ is arbitrary, using   a standard procedure based on the Kolmogorov extension theorem, one can prove the existence of a weak mild solution $X$ to \eqref{cheef} on  $[0, \infty)$. On this respect, we refer to  Remark 3.7, page 303, in \cite{KS} (cf. the beginning of Section 4).

 \vskip 2mm 
 \noindent {\bf ii)} We use Theorem \ref{extension} with $H = \H$, $A = \A$ and $W = \W$. 
 Indeed,  \eqref{cheef} can be rewritten as
 \begin{gather*}
 \label{chee}
 \begin{array}{l}
X_t = e^{t \A }x +\int_{0}^{t}(- \A)^{1/2}e^{(  t-s)  \A}F  (
X_{s})ds
 +
\int_{0}^{t}e^{(  t-s)  \A}d \W_{s}, \;\;\; t \ge 0,
\end{array}
\end{gather*}
 where
 $
 F(x) = F_0(x) + (-\A )^{-1/2} B(x),$ $x \in \H
 $. The function $F : \H \to \H$ is continuous and locally bounded (cf. \eqref{f0} and \eqref{nonl}).
\end{proof}

\begin{remark} \label{che} {\rm 
Assertion (ii) in Proposition \ref{well1} cannot be deduced directly 
from  the Girsanov theorem
as in Appendix A.1 of \cite{DFPR}. To this purpose,
 one should prove that 
 $\E \Big [ e^{\frac{1}{2} \int_0^T  |B(X_s)  |_{{\mathcal H}}^2 ds}\Big]$ $  < \infty, $
for any weak mild solution $X$  to \eqref{cheef} starting at  $x \in \H$.  A sufficient condition would be   $\E \Big [ e^{\frac{1}{2} \int_0^T  |X_s  |_{{\mathcal H}}^2 ds}\Big] < \infty$. 
  It seems that such estimate does not hold under our assumptions. 
Note that since the nonlinearity of the Burgers equation grows quadratically 
 one cannot follow the proof of Proposition 22  of \cite{DFPR} to derive 
  $\E \Big [ e^{\frac{1}{2} \int_0^T  |X_s  |_{{\mathcal H}}^2 ds}\Big] < \infty$.
 }
  \end{remark}
  }

\newpage
\appendix

\section{ A further regularity result  on  the Kolmogorov equation }

When $z=0$ one can show that   $Du$ (see \eqref{s55}) belongs to the so-called  Zygmund ${\cal C}^1$-space (see \cite{CL} and \cite{LR}).
 We will provide an alternative proof which  is inspired by Lemma 2.2 in \cite{bass}. 
 The Zygmund regularity will follow by Theorem \ref{sr}, taking into account the estimate 
 \begin{equation*}
\| D^2 P_t^{(z)}\varphi \|_0
  \le
C t^{-1}\|\varphi\|_0,\;\; t >0,\; \varphi \in B_b(H).
\end{equation*}
(see Section 1.2). 
 Let $E$ be a separable Hilbert space.   The {\sl Zygmund space}  ${ \cal C}^1 (H,E)$ is the space of  all continuous    and bounded   function $f: H \to E$, i.e., $f \in C_b (H,E)$,    such that   
\begin{equation}\label{zy} 
[f]_{{\cal C}^1}= \sup_{x,h \in H,\; h \not =0,\; |h|\le 1} \frac{|f(x+h) - 2 f(x) + f (x-h)|_E}{|h|_E } < \infty.
 \end{equation}
This  is a Banach space  endowed with the norm $\| f\|_{{\cal C}^1}
 =  [f]_{{\cal C}^1} + \| f\|_{0}$, $f \in {\cal C}^1 (H,E)$. As usual we set ${\cal C}^1(H) = {\cal C}^1(H, \R)$. 

\begin{lemma}\label{es} Let us consider a semigroup of linear contractions 
$(R_t)$, $R_t : C_b (H) \to C_b (H)$, $t\ge 0,$ such that $R_t (C_b(H)) \subset C^{2}_b(H)$, $t>0$, and there exists $C_0 >0$ such that 
\begin{equation}
 \| D^2 R_t\varphi \|_0 \le
C_0 t^{-1}\|\varphi\|_0,\;\; t \in (0,1],\; \varphi \in C_b(H).
\end{equation}
 Let $f \in C_b(H)$.  If  there exists a constant $N >0$ such that  
\begin{gather} \label{sw21}
 \sup_{x \in H}| R_tf(x) - f(x) | \, =\,
\|R_tf  - f  \|_0 \le N\,  t^{1/2}, \;\; t \in [0,1].
\end{gather}
Then $f \in {\cal C}^{1}(H)$.
Moreover, $ [f]_{{\cal C}^1}\le  16 (C_0+1)\,  (N + \|f \|_0)$.
\end{lemma}
\begin{proof} {\it I Step.} {\sl We introduce the semigroup $(\hat R_t)$,
 $
 \hat R_t = e^{-t} R_t,\;\;\; t \ge 0.
 $ and 
 prove that  }
 \begin{gather} \label{sw2}
 \| D^2 \hat R_t f \|_0 \le \frac{4 C_0\, (N + \|f \|_0) }{\sqrt{t}},\;\; t \in (0,1/2].
\end{gather}
 First note that 
 $|\hat R_t f(x)  - R_t f(x) + R_t f(x) - f(x) | $ $\le |1- e^{-t}|\| R_t f \|_0  + N t^{1/2}$, $t \in [0,1]$, and so (using also that $\|\hat R_tf  - f  \|_0 \le 2 \| f\|_0$)
 \begin{gather*}
\|\hat R_tf  - f  \|_0 \le \, (N + 2 \| f\|_0) \,  t^{1/2}, \;\; t \ge 0.
\end{gather*}
Let now $\varphi \in  C_b(H)$, $t>1$. We can write $D^2 \hat 
R_t \varphi = D^2 \hat R_1 \hat R_{t-1} \varphi$ and so 
\begin{gather*}
\| D^2 \hat R_t \varphi \|_0 \le C_0 e^{-t} \|\varphi\|_0,\;\; t>1.
\end{gather*}
It follows that,  for any $\varphi \in  C_b(H)$, 
\begin{equation}
\label{e611}  \| D^2 \hat R_t \varphi \|_0 \le
C_0 t^{-1}\|\varphi\|_0,\;\; t >0.
\end{equation}
Now by \eqref{e611} we obtain \eqref{sw2} as follows.


Let $x \in H$. Let  $k \ge 0$ be an integer and fix $t \in (0,1/2]$.  Using the semigroup law and \eqref{e611},  we write
\begin{gather*} 
 \|D^2 \hat R_{t 2^{k+1}}f(x) -  D^2 \hat R_{t 2^{k}} f(x ) \|_{\cal L}  
=
\| D^2 \hat R_{t 2^{k} + t 2^{k}}f(x) -  D^2 \hat R_{t 2^{k}} f(x ) \|_{\cal L}   
\\
= \|D^2 \hat R_{t 2^{k}}
 [ \hat R_{t 2^k} f -   f](x )\|_{\cal L}  
 \\
 \le \frac{C_0}{t 2^{k}} \|  \hat R_{t 2^k} f -   f\|_0 \le \frac{C_0 (N + 2 \| f\|_0)\, t^{1/2} {2^{k/2}} }{ 2^{k}t} 
  = \frac{C_0(N + 2 \| f\|_0)}{t^{1/2} \,{2^{k/2}} }.
\end{gather*}
Now $D^2 \hat R_{t 2^{N+1}}f(x) - D^2 \hat R_{t }f(x) =  \sum_{k = 0}^{N}
[ D^2 \hat R_{t 2^{k+1}} f(x)  - D^2 \hat R_{t 2^{k}} f(x ) ]$.
 Since we know by \eqref{e611} that $\lim_{N \to \infty} D^2 \hat R_{t 2^{N+1}}f(x) =0$, for any $x \in H$, we obtain
\begin{gather*}
  D^2 \hat R_{t} f(x) =
   \sum_{k \ge 0}
[D^2 \hat R_{t 2^{k}} f(x ) - D^2 \hat R_{t 2^{k+1}} f(x)   ], \;\; x \in H,
\end{gather*}
and we deduce \eqref{sw2} since
$ 
\sup_{x \in H}\| D^2 \hat R_{t} f(x) \|_{\cal L}  $ $ \le \frac{C_0(N + 2 \| f\|_0)}{t^{1/2} } \sum_{k \ge 0} \frac{1}{2^{k/2} }.
$
 Formula \eqref{sw2} implies, for any $t \in (0,1]$, 
\begin{gather} \label{sw}
 \| D^2 R_t f \|_0 \le \, 16 C_0 (N + \|f \|_0 ) \, t^{-1/2},\;\;\;\;  t \in (0,1].
\end{gather}
 {\it II Step. } {\sl Let us  check that $f \in {\cal C}^1 (H)$ using \eqref{sw}. } 
 \\
 Fix $h \in H$ with $|h|_H \le 1$ and  set $t = |h|^2_H$. We write
 $
f $ $= [f- R_t f] + R_t f$ $  = l_t + g_t.  
$
Since $ \| l_t\|_0 = \|f- R_t f \|_0 \le N |h|_H$ we consider $g_t$. 
Setting 
$$
\triangle_h f(x) =  f(x+h) - 2 f(x) + f (x-h),
$$
we get $
 \triangle_h f (x) = \triangle_h l_t(x) +  \triangle_h g_t(x)
$
 and $\| \triangle_h l_t \|_0 \le 4 N |h|_H.$ By the Taylor formula and \eqref{sw} we find 
\begin{gather*}
 | \triangle_h g_t(x) |_E \le \| D^2 R_t f \|_0 \, |h|^2_H \le \frac{16 C_0 
(N + \|f \|_0)
}{|h|_H} |h|^2_H = 16 C_0 (N + \|f \|_0)|h|_H. 
\end{gather*}
Hence $[f]_{{\cal C}^1} \le  16 (C_0+1)\, (N + \|f \|_0)$.  
\end{proof} 

Combining the previous lemma and Theorem \ref{sr}  we find that
  $Du^{(z)} \in {\cal C}^1(H,H)$ with a bound on  $[ Du^{(z)} ]_{{\cal C}^1}$ independent of $z$. 
\begin{theorem}\label{d} Let $f \in B_b(H)$ and consider $ Du^{(z)}$ given  in \eqref{s55}. Then with 
 $
  c_1 = 16 [C_1^2 +1]  [ C_2 +1 ] $ ($C_1$ and $C_2$ are the same of  Theorem \ref{sr})    we have 
  \begin{equation} \label{s2}
|Du^{(z)}(x+k) - 2 Du^{(z)}(x) + Du^{(z)}(x-k)|_H \le  c_1 \,   \|f \|_{0}  
 ,\;\; x \in H,\; k \in H,\; |k| \le 1,
\end{equation}
\end{theorem} 
\begin{proof} Recall that   by \eqref{wdc}
  $  \sup_{x \in H}
| D^2 P_t^{(z)}\varphi(x)|
 = \| D^2 P_t^{(z)}\varphi \|_0 \le \sqrt{2}
C_1^2 \,  t^{-1}\|\varphi\|_0,$ $ t \in (0,1],$ $ \varphi \in C_b(H)$
 %
  (the estimates holds more generally for any $t>0$).   
  We know by Theorem  \ref{sr}     that   
 \begin{gather*} 
|{P_s^{(z)}} (\langle D u^{(z)}(\cdot ), h \rangle)(x)  - \langle D u^{(z)}(x ), h \rangle |  \le C_2 |h| s^{1/2} \|f \|_{0}.  
\end{gather*} 
We obtain by Lemma  \ref{es}, for any $k \in H$ with $|k|_H \le 1$,
\begin{gather*}
|\langle D u^{(z)}(x+k) - 2 Du^{(z)}(x) + Du^{(z)}(x-k), h \rangle | \le c_1  \, |h|_H \|f \|_{0}. 
 \end{gather*}
After taking the supremum over $\{ |h|_H \le 1\}$ we obtain the assertion.
  \end{proof}
  
 
\begin{remark} \label{dai} {\em Let us consider the Ornstein-Uhlenbeck semigroup $(P_t)$ (i.e., $(P_t^{(z)})$ when $z=0$). One may ask   if a kind of  converse of Lemma \ref{es} holds. In other word if  $g$ belongs  to the Zygmund space  $ {\cal C}^{1}(H)$ we may ask if  $g$ verifies
\begin{equation}\label{magari}
\sup_{s \in (0,1)} \,  s^{-1/2}  \| P_s g - g \|_0 \,  < \infty. 
\end{equation}
 Arguing as in the proof of    Lemma 3.6 and Proposition 3.7 of \cite{DL} (considering  $\theta =1/2$ in such results) one can prove that if  $g \in {\cal C}^{1}(H)$ and in addition we have
\begin{equation}\label{ma}
\sup_{s \in (0,1)}\,   s^{-1/2} \|  g(e^{sA} (\cdot) ) - g\|_0  < \infty 
\end{equation}
then \eqref{magari} holds.
We point out that in infinite dimensions  under Hypotheses \ref{base} it is not clear if \eqref{ma} holds when $g$ is replaced by the derivative $Du $  ($Du$ is given in  \eqref{s55} with $z=0$ and $f \in B_b(H)$). 
    }
 \end{remark}

 \newpage 
 
 \noindent \textbf{Acknowledgement.}
The author would like to thank the anonymous  referees
 for their
 useful  comments and suggestions.

\end{document}